\documentclass[12pt]{amsart}

\usepackage{amsfonts,amsmath,amsthm,amssymb}
\usepackage[mathscr]{eucal}
\usepackage{tikz}
\usetikzlibrary{snakes} 
\usetikzlibrary{decorations} 

\usepackage{float}

\usepackage{xcolor}

\newtheorem{thm}{Theorem}[section]
\newtheorem{cor}[thm]{Corollary}
\newtheorem{lm}[thm]{Lemma}

\newtheorem{clm}[thm]{Claim}

\theoremstyle{definition}
\newtheorem{df}[thm]{Definition}

\theoremstyle{remark}
\newtheorem*{rem}{Remarks}

\theoremstyle{plain}

\theoremstyle{definition}

\numberwithin{equation}{section}

\newcommand{\set}[2]{\{#1\,:\,\text{#2}\}} % { x : ... }
\newcommand{\m}[1]{{\mathbf{\uppercase{#1}}}}

\DeclareMathOperator{\Con}{Con}

\newcommand{\ra}{\rightarrow}
\newcommand{\tup}[1]{\mathbf{#1}}

\newcommand{\sfb}{\mathsf{b}}
\newcommand{\sfg}{\mathsf{g}}
\newcommand{\suc}{\operatorname{suc}}
\newcommand{\baru}{\overline{u}}

\newcommand{\bars}{\overline{s}}
\newcommand{\bart}{\overline{t}}

\newcommand{\Mast}[2]{\Delta(#1,#2)}   

\newcommand{\varV}{\mathcal V}                % variety V
\newcommand{\stackalpha}{\stackrel{\alpha}{\equiv}}
\newcommand{\stackbeta}{\stackrel{\beta}{\equiv}}
\newcommand{\stacktheta}{\stackrel{\theta}{\equiv}}

\newcommand{\ol}[1]{\overline{#1}}
\newcommand{\bmat}[4]{\begin{bmatrix}#1&#3\\#2&#4\end{bmatrix}}

\def\greencol{blue}

\begin{document}
\bibliographystyle{siam}

\title{Characterizing the commutator in varieties with a difference term}
\author{Keith A. Kearnes}
\address{Department of Mathematics\\
University of Colorado\\
Boulder, Colorado  80309-0395\\
USA}
\email{kearnes@colorado.edu}

\author{\'{A}gnes Szendrei}
\address{Department of Mathematics\\
University of Colorado\\
Boulder, Colorado 80309-0395\\
USA}
\email{szendrei@colorado.edu}

\author{Ross Willard}
\address{Pure Mathematics Department\\
University of Waterloo\\
Waterloo, Ontario N2L 3G1\\
Canada}
\email{ross.willard@uwaterloo.ca}

\thanks{This material is based upon work supported by the National Research, Development and Innovation Fund of Hungary (NKFI) grant no. K128042, and a
Natural Sciences and Engineering Research Council (NSERC) of Canada Discovery 
Grant.}



\date{February 15, 2022}

\begin{abstract}
We extend the validity of Kiss's characterization of
``$[\alpha,\beta]=0$''
from congruence modular varieties to varieties with a difference 
term.  This fixes a recently discovered gap in our paper \cite{ksw2016}.
We also prove some related properties of Kiss terms in varieties with a
difference term.
\end{abstract}

\maketitle

\section{Introduction}
In \cite{gumm1980},
H.~Peter~Gumm characterized the commutator relation ``$[\alpha,\beta]=0$''
when $\alpha$ and $\beta$ are comparable congruences ($\alpha\leq \beta$)
on an algebra in a congruence modular variety.
His characterization involved the concept of a
$3$-ary difference term (or ``Gumm term''). 
A $3$-ary difference term for a variety $\varV$
is a $3$-ary term $p(x,y,z)$ such that 
\begin{enumerate}
\item[(I)${}_p$]  $p(x,x,y)\approx y$ holds in every
  member of $\varV$, while
\item[(II)${}_p$]  $p(a,b,b)\stackrel{[\theta,\theta]}{\equiv} a$ whenever
  $(a,b) \in \theta \in \Con(\m a)$ for some $\m a\in \varV$.
  \end{enumerate}
Gumm showed that (i) every congruence modular variety
has a $3$-ary difference term, while
it is easy to verify
(using \cite[Corollary 4.7]{ks1998})
that (ii) $p(x,y,z)=z$
is a $3$-ary difference term for 
any congruence meet-semidistribu\-tive variety.
The class of all varieties having a difference term is a natural class to 
study: it is intermediate between the class of congruence modular varieties
and the class of Taylor varieties, it is definable by a (linear)
Maltsev condition \cite{lip1996,ks1998,ksw2016},
and the commutator operation is fairly well-behaved in this class
\cite{lip1994,kea1995,ks1998} --
in particular, the commutator operation is commutative 
in varieties with a difference term \cite[Lemma 2.2]{kea1995}.

In \cite[Definition 3.3]{kiss1992},
Emil Kiss introduced the concept of a
$4$-ary difference term
for congruence modular varieties. He used this concept to
characterize the relation ``$[\alpha,\beta]=0$''
when $\alpha$ and $\beta$ are not-necessarily-comparable congruences
on an algebra in a congruence modular variety.
The definition of a $4$-ary difference term refers
to a $4$-ary relation $R(\alpha,\beta)$ defined
from $\alpha, \beta\in\Con(\m a)$, $\m a\in \varV$ as follows:
\[
R(\alpha,\beta) = \set{(a,b,c,d) \in A^4}{$
(a,b),(c,d) \in \alpha$ 
and $(a,c),(b,d) \in \beta$}.
\]
$R(\alpha,\beta)$ is a subalgebra of $\m a^4$.
A $4$-ary difference term (or ``Kiss term'')
for a variety $\varV$
is a $4$-ary term $q(x,y,z,w)$ such that 
\begin{enumerate}
\item[(I)${}_q$]  $q(x,x,y,y)\approx y$ and $q(x,y,x,y)\approx x$
  hold in every
  member of $\varV$, while
\item[(II)${}_q$]  $q(a,b,c,d) \stackrel{[\alpha,\beta]}{\equiv} q(a,b,c',d)$
  holds whenever $\alpha, \beta\in \Con(\m a)$,
  $\m a\in \varV$, and
  $(a,b,c,d), (a,b,c',d)\in R(\alpha,\beta)$.
  \end{enumerate}
Kiss shows in his paper that (i) every congruence modular variety
has a $4$-ary difference term, while 
it is easy to verify
that (ii) $q(x,y,z,w)=z$ is a $4$-ary difference term for 
any congruence meet-semidistributive variety. 

At the top of page 467 of \cite{kiss1992},
Kiss states that if $q(x,y,z,w)$ is a $4$-ary
difference term for a congruence modular
variety $\varV$, then $p(x,y,z):=q(x,y,z,z)$
will be a $3$-ary difference term for $\varV$.
We give the argument for this in the next paragraph
in order to illustrate that this statement does
not require
$\varV$ to be congruence modular.

Assume that $q(x,y,z,w)$ is a $4$-ary difference term
for $\varV$ and define $p(x,y,z)=q(x,y,z,z)$.
From the identity
$q(x,x,y,y) \approx y$, which is part of (I)${}_q$, we derive that
$p(x,x,y)=q(x,x,y,y) \approx y$ holds in every
member of $\varV$, which is the claim of (I)${}_p$.
That is, (I)${}_q$ suffices to prove (I)${}_p$
if $p(x,y,z)=q(x,y,z,z)$.
Now assume that $\m a\in \varV$ and
$(a,b) \in \theta \in \Con(\m a)$.
Let $\alpha = \beta=\theta$.
Then $(A,B,C,D):=(a,b,b,b)$
and $(A,B,C',D):=(a,b,a,b)$
both belong to $R(\alpha,\beta)=R(\theta,\theta)$.
From the defining properties (I)${}_q$ and
(II)${}_q$ of $q$ we get
\[
p(a,b,b)=
q(a,b,b,b)=q(A,B,C,D) \stackrel{[\theta,\theta]}{\equiv} q(A,B,C',D)=
q(a,b,a,b) = a.
\]
This is the claim of (II)${}_p$.
That is, (I)${}_q$ and (II)${}_q$ 
suffice to prove (II)${}_p$
if $p(x,y,z)=q(x,y,z,z)$.

In Problem~3.11 of his paper, Kiss asks whether, conversely,
there is a reasonable way to
construct a $4$-ary difference term from a $3$-ary difference
term.
Paolo Lipparini solved this problem in \cite{lip1999},
by showing that, for any variety,
\begin{equation*}\label{LipForm}
  q(x,y,z,w) := p(p(x,z,z),p(y,w,z),z)
\tag{L}
\end{equation*}
is a $4$-ary difference term
whenever $p(x,y,z)$ is a $3$-ary difference term.
This shows that a variety has a $3$-ary difference term
if and only if it has a $4$-ary difference term,
and we will refer to such a variety as a
``variety with a difference term''.
We will reserve the
phrase ``difference term'' for a $3$-ary difference term $p$
and use the phrase ``Kiss term'' for a $4$-ary difference term $q$.

Our goal in this paper is to show that
Kiss's characterization of 
``$[\alpha,\beta]=0$'' holds for any variety with a
difference term, whether the variety is congruence modular or not.

In fact, we already claimed to have done this in
\cite[Lemma 6.2]{ksw2016}. But recently,
Ralph Freese and Peter Mayr
discovered that the proof in \cite{ksw2016}
of Lemma~6.2 has a gap.
Lemma~6.2 of \cite{ksw2016} states (under the assumption
that $\varV$ has a difference term and $q$ is the Kiss term obtained
from it via Lipparini's Formula \eqref{LipForm}):

\medskip\noindent
\textbf{Lemma 6.2 of \cite{ksw2016}.} 
\emph{If $\m a \in \varV$ and $\alpha,\beta \in \Con(\m a)$, then
$[\alpha,\beta]=0$ iff 
\begin{enumerate}
\item[(i)]
$q\colon R(\alpha,\beta)\ra \m a$ is a homomorphism, and
\item[(ii)]
$q$ is independent of its third variable on $R(\alpha,\beta)$.
\end{enumerate}
}

The incomplete part of the proof of Lemma~6.2 in \cite{ksw2016} 
concerns the implication
\begin{center}
$[\alpha,\beta]=0$ \; implies\; Item~(i).
\end{center}
In this paper we will fill the gap with

\begin{thm} \label{thm:main}
If $\varV$ is a variety with difference term $p$,
$q(x,y,z,w)$ is the Kiss term obtained from $p$ via
{\rm (\ref{LipForm})}, 
$\m a \in \varV$, $\alpha,\beta \in \Con(\m a)$, and $[\alpha,\beta]=0$,
then $q\colon R(\alpha,\beta)\ra \m a$ is a homomorphism.
\end{thm}

Then in Section~\ref{sec:Lm6.2} we will restate Lemma~6.2 of \cite{ksw2016}
and prove it.

Although Lemma~6.2 of \cite{ksw2016} only makes a claim
about the Kiss term $q$ obtained from the
difference term $p$ by Lipparini's Formula \eqref{LipForm},
we shall prove in Section~\ref{section:refinements}
that the statement of Lemma~6.2 holds
for any Kiss term in any variety that has a Kiss term.
Section~\ref{section:refinements} contains other refinements and
extensions of results in earlier sections.

\section{High-level summary of the proof of Theorem~\ref{thm:main}} \label{sec:summary}

We represent a 4-tuple $(a,b,c,d)\in A^4$ in matrix form as
\[
\begin{bmatrix}
a&c\\b&d\end{bmatrix}.
\]
Thus a matrix in $R(\alpha,\beta)$ has its columns in $\alpha$ and
its rows in $\beta$.
To emphasize this representation, we shall henceforth rename
$\m a^4$
as $\m a^{2\times 2}$.
As usual, we define $M(\alpha,\beta)$ to be the subalgebra of 
$\m a^{2\times 2}$ generated by 
\begin{equation} \label{eq:G}
G(\alpha,\beta) :=
\left\{ \begin{bmatrix} c&c\\d&d\end{bmatrix}
: (c,d) \in \alpha\right\}
\cup
\left\{
\begin{bmatrix} a&c\\a&c\end{bmatrix}
: (a,c) \in \beta\right\}.
\end{equation}

In the congruence modular setting, the ``horizontal transitive closure''
of $M(\alpha,\beta)$, denoted $\Delta_{\alpha,\beta}$, plays a key role.
In that setting it turns out that $\Delta_{\alpha,\beta}$ is also 
``vertically transitively closed," and hence $\Delta_{\beta,\alpha}$ is
equal to the set of transposes of the matrices in $\Delta_{\alpha,\beta}$.
These facts can fail outside the congruence modular setting.  
Andrew Moorhead \cite{moor2021} recently identified and studied the 
``horizontal and vertical
transitive closure" of $M(\alpha,\beta)$ in a general setting.  This
4-ary relation is an example of a 
construct that Moorhead calls a ``2-dimensional congruence," and 
George Janelidze and M. Cristina Pedicchio \cite{janped2001} previously called 
a ``double equivalence relation."
Moorhead's ``horizontal and vertical transitive closure" of $M(\alpha,\beta)$
will play a crucial role in our arguments.

\begin{df}[{\cite[Definition 2.10 and Lemma 2.13]{moor2021}}]\label{df:delta}
Let $\m a$ be an algebra and $\alpha,\beta \in \Con(\m a)$.
$\Mast\alpha\beta$ denotes the ``horizontal and vertical'' transitive
closure of $M(\alpha,\beta)$.  That is, $\Mast\alpha\beta$ is the 
smallest subset of $R(\alpha,\beta)$ containing $M(\alpha,\beta)$ and satisfying
the following two closure conditions:
\begin{enumerate}
\item (Horizontal gluing)
If $\begin{bmatrix}a&a'\\b&b'\end{bmatrix}, 
\begin{bmatrix}a'&a''\\b'&b''\end{bmatrix} \in \Mast\alpha\beta$,
then 
$\begin{bmatrix}a&a''\\b&b''\end{bmatrix} \in \Mast\alpha\beta$.\\[1mm]
\item (Vertical gluing)
If $\begin{bmatrix}a&a'\\b&b'\end{bmatrix}, 
\begin{bmatrix}b&b'\\c&c'\end{bmatrix} \in \Mast\alpha\beta$,
then 
$\begin{bmatrix}a&a'\\c&c'\end{bmatrix} \in \Mast\alpha\beta$.
\end{enumerate}
\end{df}

Here are the facts about $\Mast\alpha\beta$ that we need.

\begin{lm}[{\cite[Lemma 2.9]{moor2021}}] \label{lm:subalg}
For any algebra $\m a$ and congruences $\alpha,\beta \in \Con(\m a)$,
$\Mast\alpha\beta$ is a subalgebra of $\m a^{2\times 2}$.
\end{lm}

\begin{lm} \label{lm:transpose}
For any algebra $\m a$ and congruences $\alpha,\beta \in \Con(\m a)$,
$\Mast\beta\alpha$ is the set of transposes of matrices in $\Mast\alpha\beta$.
\end{lm}

\begin{lm}  \label{lm:C2*}
Suppose $\m a$ 
is an algebra in a variety with a difference term
and $\alpha,\beta \in\Con(\m a)$ with $[\alpha,\beta]=0$.
\begin{enumerate}
\item
If
$\begin{bmatrix}a&c\\b&d\end{bmatrix} \in \Mast\alpha\beta$, then
$a=c$ iff $b=d$.\\[1mm]
\item
If 
$\begin{bmatrix}a&c\\b&d\end{bmatrix},
\begin{bmatrix}a&c'\\b&d\end{bmatrix} \in \Mast\alpha\beta$,
then $c=c'$.
\end{enumerate}
\end{lm}

\begin{lm} \label{lm:crucial}
Let $\m a$ be an algebra in a variety having a
difference term $p(x,y,z)$ and let
$q(x,y,z,w)$ be the Kiss term obtained from $p$ via
\rm{(\ref{LipForm})}. 
If $\alpha, \beta\in\Con(\m a)$ with $[\alpha,\beta]=0$, then
for all $\begin{bmatrix}a&c\\b&d
\end{bmatrix} \in R(\alpha,\beta)$ 
we have $\begin{bmatrix} a&c'\\b&d\end{bmatrix} \in \Mast\alpha\beta$
where $c'=q(a,b,c,d)$.
\end{lm}

Lemmas~\ref{lm:subalg}--\ref{lm:C2*} are easy to prove.
Theorem~\ref{thm:main} follows easily from
Lemmas~\ref{lm:subalg}--\ref{lm:crucial},
essentially following our (incorrect) proof in \cite{ksw2016} except
replacing $\Delta_{\alpha,\beta}$ at every step with $\Mast\alpha\beta$.
We shall prove Lemmas~\ref{lm:subalg}--\ref{lm:C2*}
in the next section.

Lemma~\ref{lm:crucial} is not obvious.
For example, in the congruence meet-semidis\-tributive case
(where $p(x,y,z)$ can be taken to be $z$) it implies
the nonobvious fact that
if $\alpha\cap\beta=0$ then $R(\alpha,\beta)=\Mast\alpha\beta$.
Note that in congruence
meet-semidistributive varieties, Moorhead's analysis yields
$R(\alpha,\alpha)= \Mast\alpha\alpha$ for any congruence $\alpha$
\cite[Theorem 5.2]{moor2021}; our Lemma~\ref{lm:crucial}
applied to the
congruence meet-semidistributive case may be seen as extending 
Moorhead's result to pairs of disjoint congruences.
We will prove Lemma~\ref{lm:crucial} 
in Sections~\ref{sec:reduction}
and~\ref{sec:laststep} using our Maltsev condition for varieties with a 
difference term.  The proof of Theorem~\ref{thm:main} is completed in
Section~\ref{section:assembly}.

\section{Proofs of Lemmas~\ref{lm:subalg}--\ref{lm:C2*}} \label{section:easyLM}

\begin{proof}[Proof of Lemma~\ref{lm:subalg}]
This is a special case of \cite[Lemma 2.9~(2)]{moor2021}.  However, 
we include an alternative proof here because we will need some of its 
elements in our proofs of Lemma~\ref{lm:C2*} and Theorem~\ref{thm:quotient}.
Fix $\m a$ and congruences $\alpha,\beta$.  For $n \in \omega$ define
$M_n(\alpha,\beta)$ by setting $M_0(\alpha,\beta)=M(\alpha,\beta)$,
and for $n>0$, 
\begin{itemize}
\item
If $n$ is odd, then a matrix is in $M_n(\alpha,\beta)$ iff it
can be realized as the result of horizontally
gluing two matrices from $M_{n-1}(\alpha,\beta)$; that is, 
\[
\begin{bmatrix}a&c\\b&d \end{bmatrix}
\in M_n(\alpha,\beta)\quad\mbox{iff}\quad
\exists r,s \in A~\mbox{with}~
\begin{bmatrix}a&r\\b&s\end{bmatrix},
\begin{bmatrix}r&c\\s&d\end{bmatrix} \in M_{n-1}(\alpha,\beta).
\]
\item
If $n$ is even, then a matrix is in $M_n(\alpha,\beta)$ iff it
can be realized as the result of vertically 
gluing two matrices from $M_{n-1}(\alpha,\beta)$; that is, 
\[
\begin{bmatrix}a&c\\b&d \end{bmatrix}
\in M_n(\alpha,\beta)\quad\mbox{iff}\quad
\exists r,s \in A~\mbox{with}~
\begin{bmatrix}a&c\\r&s\end{bmatrix},
\begin{bmatrix}r&s\\b&d\end{bmatrix} \in M_{n-1}(\alpha,\beta).
\]
\end{itemize}

One can easily prove by induction that for each $n\geq 0$,
\begin{enumerate}
\item[(i)]
$M_n(\alpha,\beta)$ is a subalgebra
of $\m a^{2\times 2}$,
\item[(ii)]
$M_n(\alpha,\beta)\subseteq \Mast\alpha\beta$, 
\item[(iii)]
$M_n(\alpha,\beta)\subseteq M_{n+1}(\alpha,\beta)$. 
\end{enumerate}
Hence if we let $M_\omega(\alpha,\beta) = \bigcup_{n=0}^\infty M_n(\alpha,
\beta)$, then $M_\omega(\alpha,\beta)$ 
is a subalgebra of $\m a^{2\times 2}$
and $M_\omega(\alpha,\beta)\subseteq\Mast\alpha\beta$.
For the opposite inclusion,
observe that if a matrix $B$ can be obtained from two matrices $C,D \in
M_\omega(\alpha,\beta)$
by gluing (either horizontally or vertically), then by (iii) 
there exists $n \in \omega$ such that $C,D \in M_n(\alpha,\beta)$ and
hence $B \in M_{n+2}(\alpha,\beta)$.  
This proves that $M_\omega(\alpha,\beta)$ 
is closed under horizontal and vertical transitive closure,
which implies $\Mast\alpha\beta\subseteq M_\omega(\alpha,\beta)$.
\end{proof}

\begin{proof}[Proof of Lemma~\ref{lm:transpose}]
This 
follows easily from the symmetry between horizontal and vertical transitive
closures in the definition of $\Mast\alpha\beta$, 
and the fact that the set of generators for
$M(\beta,\alpha)$ is the set of transposes of the generators for
$M(\alpha,\beta)$.
\end{proof}

\begin{proof}[Proof of Lemma~\ref{lm:C2*}]
This can be deduced from \cite[Lemma 2.2]{kea1995} and
\cite[Corollary 4.5]{ks1998} and the description of the linear commutator
in \cite{ks1998}.  For completeness, we give a short direct proof here.
Fix 
$\m a$ and $\alpha,\beta \in \Con(\m a)$ with $[\alpha,\beta]=0$.
To prove (1), we recycle the notation from the proof of Lemma~\ref{lm:subalg}
and show by induction on $n$ that if
$\begin{bmatrix}a&c\\b&d\end{bmatrix} \in M_n(\alpha,\beta)$ then
$a=c$ iff $b=d$.  This claim follows automatically from $[\alpha,\beta]=0$
if $n=0$, so assume $n>0$.  Assume without loss of
generality that $a=c$.  If $n$ is even, then there exist $r,s \in A$ with
\[
\begin{bmatrix}a&a\\r&s\end{bmatrix},
\begin{bmatrix}r&s\\b&d\end{bmatrix} \in M_{n-1}(\alpha,\beta).
\]
By the inductive hypothesis 
we get $r=s$ from the first matrix and then $b=d$ from
the second matrix as required.  If instead $n$ is odd, then there
exist $r,s \in A$ with
\[
\begin{bmatrix}a&r\\b&s\end{bmatrix},
\begin{bmatrix}r&a\\s&d\end{bmatrix} \in M_{n-1}(\alpha,\beta).
\]
One can show that each $M_k(\alpha,\beta)$ is closed under swapping
columns, and we use that here to get 
$\begin{bmatrix} r&a\\s&b \end{bmatrix}\in M_{n-1}(\alpha,\beta)$.
Also note that our assumptions imply $(b,d) \in \alpha\cap\beta$, which
with $[\alpha,\beta]=0$ implies $p(d,b,b)=d$.  Now use the usual trick:
apply $p$ to the following three matrices from $M_{n-1}(\alpha,\beta)$:
\[
\begin{bmatrix} r&a\\s&d\end{bmatrix},
\begin{bmatrix} r&a\\s&b\end{bmatrix},
\begin{bmatrix} b&b\\b&b\end{bmatrix}.
\]
The resulting matrix $\begin{bmatrix}b&b\\b&d\end{bmatrix}$ is again
  in $M_{n-1}(\alpha,\beta)$ (since $M_k(\alpha,\beta)$ is a subalgebra for
  any $k$).
  We can apply the inductive hypothesis
  to this last matrix to get $b=d$, which was our aim.

(2) Suppose
\[
\begin{bmatrix}a&c\\b&d\end{bmatrix},
\begin{bmatrix}a&c'\\b&d\end{bmatrix} \in \Mast\alpha\beta.
\]
Swapping the two columns of the first matrix gives
$\begin{bmatrix}c&a\\d&b\end{bmatrix} \in \Mast\alpha\beta$.  This last
matrix can be glued horizontally to the second matrix above to get a matrix
in $\Mast\alpha\beta$ whose entries on the bottom row are equal, 
and having $c,c'$ as the
entries on the top row.  We then get $c=c'$ from part (1).
\end{proof}

\section{Proof of Lemma~\ref{lm:crucial}: reduction} \label{sec:reduction}

Throughout this section we fix a variety $\varV$ having
a difference term
$p(x,y,z)$.
Let $\m f = \mathbb F_{\varV}(x,y,z,w)$ be the free algebra over $\{x,y,z,w\}$
in $\varV$.  We will consider a certain set $T$ of 
$2\times 2$ matrices with entries in $F$, defined as follows: if
\[
M =
\begin{bmatrix}t_1 & t_3\\
  t_2 &t_4\end{bmatrix}
=
\begin{bmatrix}t_1(x,y,z,w)  & t_3(x,y,z,w)\\
  t_2(x,y,z,w) &t_4(x,y,z,w)\end{bmatrix}
  \in F^{2\times 2},
  \]
then
\begin{align*}
M \in T &\iff \textrm{For all } \m a\in \varV,~\textrm{for all } \alpha,\beta \in \Con(\m a)~\mbox{with}~
[\alpha,\beta]=0, \textrm{ and}\\
&\textrm{for all} \begin{bmatrix}a&c\\b&d \end{bmatrix} \in R(\alpha,\beta)~
\mbox{we have}~ \begin{bmatrix}t_1(a,b,c,d) & t_3(a,b,c,d)\\
t_2(a,b,c,d) & t_4(a,b,c,d) \end{bmatrix} \in \Mast\alpha\beta.
\end{align*}

Our ultimate goal is to prove Lemma~\ref{lm:crucial}.  Note that this is 
equivalent to proving 
$\begin{bmatrix}x & q\\y & w\end{bmatrix} \in T$
where $q$ is the Kiss term obtained from $p$ by \rm{(\ref{LipForm})}.

In what follows, for readability, we will frequently suppress commas
in rendering terms. For example, 
we may write $p(xzz)$ for $p(x,z,z)$.

\begin{lm} \label{lm:agi}
If $\begin{bmatrix} x & p(xzz)\\ y & p(yww)\end{bmatrix} \in T$,
then 
$\begin{bmatrix}x & q(xyzw)\\y & w\end{bmatrix} \in T$.
\end{lm}

\begin{proof}
Assume that
$\m a \in \varV$, $\alpha,\beta \in \Con(\m a)$ with $[\alpha,\beta]=0$,
and
\begin{equation} \label{abcdR}
\begin{bmatrix}a&c\\b&d\end{bmatrix} \in R(\alpha,\beta).
\end{equation}
The assumption of the lemma statement together
with the assumptions of the first
line of the proof imply that
\begin{equation} \label{M0}
M_0 := \begin{bmatrix} a & p(acc)\\b & p(bdd)\end{bmatrix}
\in \Mast\alpha\beta.
\end{equation}
Since $p(xxx)\approx x$,
$(a,b), (c,d)\in\alpha$,  and $(a,c), (b,d)\in\beta$,
we have 
\begin{equation} \label{pacc}
\begin{bmatrix} p(acc) & c\\
  p(bdd) & d \end{bmatrix}=
\begin{bmatrix} p(acc) & p(ccc)\\
  p(bdd) & p(ddd) \end{bmatrix}
\in R(\alpha,\beta).
\end{equation}
Applying to (\ref{pacc})
the argument that led from the matrix
in (\ref{abcdR}) to the matrix in (\ref{M0}),
we obtain 
\begin{equation} \label{M1}
M_1:= \begin{bmatrix} p(acc) & p(p(acc)cc)\\ p(bdd) & p(p(bdd)dd) \end{bmatrix}
\in \Mast\alpha\beta.
\end{equation}
Finally, the following three matrices 
\[
\begin{bmatrix} p(acc) & p(acc)\\
p(bdd) & p(bdd) \end{bmatrix},\quad
\begin{bmatrix} p(bbc) & p(bdc)\\
p(bbd) & p(bdd) \end{bmatrix},\quad
\begin{bmatrix} c & c\\
d & d \end{bmatrix}
\]
are in $M(\alpha,\beta)$ because
$(a,b), (c,d)\in\alpha$  and $(a,c), (b,d)\in\beta$,
so applying $p$ to them we get that the matrix 
\[
M_2 := 
\begin{bmatrix}
p(p(acc)cc) & q(abcd)\\
p(p(bdd)dd) & d
\end{bmatrix}
=
\begin{bmatrix}
p(p(acc),p(bbc),c) & p(p(acc),p(bdc),c)\\
p(p(bdd),p(bbd),d) & p(p(bdd),p(bdd),d)
\end{bmatrix}
\]
is in $M(\alpha,\beta)$ and hence is in $\Mast\alpha\beta$.
Gluing $M_0$, $M_1$ and $M_2$ horizontally gives
\[
\begin{bmatrix} a & q(abcd)\\ b & d \end{bmatrix}
\in \Mast\alpha\beta. 
\]
Since this is true for arbitrary $\m a \in \varV$, $\alpha,\beta \in \Con(\m a)$
with $[\alpha,\beta]=0$, and arbitrary $\begin{bmatrix} a&c\\b&d \end{bmatrix}
\in R(\alpha,\beta)$, we have proved the conclusion of
the lemma.
\end{proof}

Lemma~\ref{lm:agi} reduces the task of proving Lemma~\ref{lm:crucial}
to the following:

\medskip\noindent\textbf{Remaining Goal:} to show
$\begin{bmatrix}x&p(xzz)\\y&p(yww)\end{bmatrix} \in T$.

\medskip
We will complete this remaining goal with the help of the other terms in
the Maltsev condition characterizing a difference term.
This Maltsev condition was alluded to in our paper
\cite[Theorem 1.2, proof of (2) $\Rightarrow$ (3)]{ksw2016}, but now
we need to make it explicit.  By virtue of $p$ being a difference term for
$\varV$, there exist:
\begin{itemize}
\item
a finite vertex-labeled tree $\mathcal T$ with vertex set $I$ 
and root 0,
such that the children of each nonleaf are linearly ordered,
every vertex is labeled by $\sfb$ or $\sfg$, 0 is labeled $\sfb$, and a child
always has the opposite label of its parent; and
\item
a family $\set{(f_i,g_i)}{$i \in I$}$ of pairs of idempotent
3-ary terms in $(x,y,z)$;
\end{itemize}
which, with $p$, 
satisfy some identities which will be stated after some preparation.
First,
\begin{itemize}
\item
For each nonleaf vertex $i \in I$ let $\alpha(i)$ and $\omega(i)$ 
denote the first and last children respectively of $i$;
\item
For each vertex $i \in I\setminus \{0\}$, let $\pi(i)$ denote the parent
of $i$, and if
$i\ne \omega(\pi(i))$ then let $\suc(i)$ denote the first child of $\pi(i)$
after $i$.
\end{itemize}
Then the identities are:
\begin{align}
f_i(x,y,x) &\approx g_i(x,y,x) \qquad\mbox{for all $i \in I$}\label{ax1}\\
f_0(x,y,z) &\approx x\label{ax2}\\
f_i(x,x,y) & \approx g_i(x,x,y) \qquad\mbox{if $i$ is a leaf colored $\sfb$}
\label{ax3}\\
f_i(x,y,y) & \approx g_i(x,y,y) \qquad\mbox{if $i$ is a leaf colored $\sfg$}
\label{ax4}\\
f_i(x,x,y) &\approx f_{\alpha(i)}(x,x,y) \qquad \mbox{if $i$ is a nonleaf colored $\sfb$}\label{ax5}\\
g_i(x,x,y) &\approx g_{\omega(i)}(x,x,y) \qquad \mbox{if $i$ is a nonleaf colored $\sfb$}\label{ax6}\\
f_i(x,y,y) &\approx f_{\alpha(i)}(x,y,y) \qquad \mbox{if $i$ is a nonleaf colored $\sfg$}\label{ax7}\\
g_i(x,y,y) &\approx g_{\omega(i)}(x,y,y) \qquad \mbox{if $i$ is a nonleaf colored $\sfg$}\label{ax8}\\
g_i(x,x,y) &\approx f_{\suc(i)}(x,x,y) \qquad \mbox{if $i\ne 0$, $\pi(i)$ is
colored $\sfb$, and $i\ne \omega(\pi(i))$}\label{ax9}\\
g_i(x,y,y) &\approx f_{\suc(i)}(x,y,y) \qquad \mbox{if $i\ne 0$, $\pi(i)$ is
colored $\sfg$, and $i\ne \omega(\pi(i))$}\label{ax10}\\
g_0(x,y,y) &\approx p(x,y,y)\label{ax11}\\
p(x,x,y) &\approx y. \label{ax12}
\end{align}
(Note: in 
\cite[Theorem 1.2, proof of (2) $\Rightarrow$ (3)]{ksw2016}, 
we used $q$ to denote $g_0$, but we do not do that here since we are 
using $q$ for the Kiss term.)

\begin{df} \label{df:LiRi}
For each $i \in I$, define the matrices $L_i,R_i \in F^{2\times 2}$ by
\[
L_i =
\begin{bmatrix} f_i(xxz) & g_i(xxz)\\
f_i(yyw) & g_i(yyw)\end{bmatrix}  \quad\mbox{and}\quad
R_i=\begin{bmatrix} f_i(xzz) & g_i(xzz)\\
f_i(yww) & g_i(yww)\end{bmatrix}.
\]
\end{df}

Note in particular that 
\[
R_0 = \begin{bmatrix} f_0(xzz) & g_0(xzz)\\f_0(yww) & g_0(yww)
\end{bmatrix} = \begin{bmatrix}x&p(xzz)\\y & p(yww)
\end{bmatrix}
\]
by identities \eqref{ax2} and \eqref{ax11}.
Thus we can restate our remaining goal as ``$R_0 \in T$.''

\begin{lm}  \label{lm:step1}
If for every $i \in I$ we have
$L_i \in T \iff R_i \in T$, then $R_0 \in T$.
\end{lm}

\begin{proof}
It suffices by the assumption to show that 
$T\cap \{L_0,R_0\}\ne\varnothing$.  In fact we will show
$T \cap \{L_i,R_i\}\ne \varnothing$
for all $i \in I$,  by induction on $i$ starting at the 
leaves.  If $i \in I$ is a leaf colored $\sfb$, then $L_i \in T$ by identity
\eqref{ax3}.  Similarly, if $i \in I$ is a leaf colored $\sfg$, then
$R_i \in T$ by identity \eqref{ax4}.

Next assume that $i \in I$ is not a leaf.  Let $\alpha(i)=i_1,\ldots,i_k =
\omega(i)$ be the list of children of $i$ in their prescribed order.  
By induction, we can assume that $T\cap \{L_j,R_j\}\ne \varnothing$ for each
$j=i_1,\ldots,i_k$.  Hence by the assumption we have $\{L_j,R_j\}\subseteq
T$ for each $j=i_1,\ldots,i_k$.

First consider the case when $i$ is colored $\sfb$.  
Let $\m a \in \varV$,
let $\alpha,\beta \in \Con(\m a)$ with $[\alpha,\beta]=0$, and let
$\begin{bmatrix}a&c\\b&d\end{bmatrix} \in R(\alpha,\beta)$.  Consider the 
matrices 
\[
\begin{bmatrix}f_{i_1}(aab) & g_{i_1}(aab)\\f_{i_1}(ccd) & g_{i_1}(ccd)\end{bmatrix},
\begin{bmatrix}f_{i_2}(aab) & g_{i_2}(aab)\\f_{i_2}(ccd) & g_{i_2}(ccd)\end{bmatrix},
\cdots,
\begin{bmatrix}f_{i_k}(aab) & g_{i_k}(aab)\\f_{i_k}(ccd) & g_{i_k}(ccd)\end{bmatrix}.
\]
Because $L_{i_1},\ldots,L_{i_k} \in T$, the above matrices are in 
$\Mast\alpha\beta$.  And because $\suc(i_1)=i_2$, $\suc(i_2)=i_3$ etc., 
the identities \eqref{ax9} guarantee that we can glue the above
matrices horizontally to get 
\[
\begin{bmatrix} 
f_{\alpha(i)}(aab) & g_{\omega(i)}(aab)\\
f_{\alpha(i)}(ccd) & g_{\omega(i)}(ccd)\end{bmatrix} \in \Mast\alpha\beta.
\]
Identities \eqref{ax5} and \eqref{ax6} give that this matrix equals
$\begin{bmatrix} 
f_i(aab) & g_i(aab)\\
f_i(ccd) & g_i(ccd)\end{bmatrix}$.  
This proves $L_i \in T$.
The case when $i$ is colored $\sfg$ is handled
similarly.
\end{proof}

\section{Proof of Lemma~\ref{lm:crucial}: final step} \label{sec:laststep}

We continue to assume that $\varV$ is a variety with a difference term
$p$ and other terms $f_i,g_i$ ($i \in I$) witnessing the Maltsev 
condition stated in the previous section.  By Lemmas~\ref{lm:agi} and~\ref{lm:step1}, the next theorem will finish the proof of Lemma~\ref{lm:crucial}.

\begin{thm} \label{thm:laststep}
For every $i \in I$, $L_i \in T \iff R_i \in T$.
\end{thm}

\begin{proof}
What we actually prove 
is that for any two idempotent 3-ary terms
$f(x,y,z)$, $g(x,y,z)$ satisfying $\varV\models f(x,y,x)\approx g(x,y,x)$, if
we set
\[
L_{fg} := \begin{bmatrix}f(xxz) & g(xxz)\\f(yyw) & g(yyw)\end{bmatrix}
\quad\mbox{and}\quad
R_{fg} := \begin{bmatrix}f(xzz) & g(xzz)\\f(yww) & g(yww)\end{bmatrix}
\]
then $R_{fg} \in T$ implies $L_{fg} \in T$.
This will suffice as we now explain.  Fix $i \in I$.
Applying the above claim to $(f,g) = (f_i,g_i)$  and
invoking identity \eqref{ax1} shows that $R_i \in T$ implies $L_i \in T$.
For the reverse implication, let $\delta$ be the automorphism of $\m f$
which swaps $x$ with $z$ and $y$ with $w$.  
It is easy to check that $T$ is closed under the componentwise action of
$\delta$.  Thus if $L_i \in T$ then 
\[
\delta(L_i) = \begin{bmatrix}f_i(zzx) & g_i(zzx)\\
f_i(wwy) & g_i(wwy)\end{bmatrix} \in T.
\]
Now define $f(x,y,z) := f_i(z,y,x)$ and $g(x,y,z) := g_i(z,y,x)$.  
Then $R_{fg}=\delta(L_i) \in T$.
We have $\varV\models f(xyx)\approx g(xyx)$ by identity \eqref{ax1}, so we can
apply the above claim to get $L_{fg} \in T$, and hence
$R_i = \delta(L_{fg}) \in T$ as well.

So let $f,g$ be idempotent 3-ary terms satisfying 
\begin{equation}
\varV\models f(xyx)\approx g(xyx),
\tag{$\dagger$} \label{4.1ass}
\end{equation}
and assume $R_{fg}\in T$.  
In order to prove $L_{fg}\in T$, we must show that
\[
\begin{bmatrix} f(aac) & g(aac)\\f(bbd) & g(bbd) \end{bmatrix}
\in \Mast\alpha\beta
\]
whenever $\m a \in \varV$, $\alpha,\beta \in \Con(\m a)$
with $[\alpha,\beta]=0$,
and $\begin{bmatrix}a&c\\b&d \end{bmatrix} \in R(\alpha,\beta)$.

To this end, choose and fix arbitrary
$\m a \in \varV$, $\alpha,\beta \in \Con(\m a)$
with $[\alpha,\beta]=0$,
and $\begin{bmatrix}a&c\\b&d \end{bmatrix} \in R(\alpha,\beta)$.
We start by defining some matrices in $\Mast\alpha\beta$.  
Let
\[
B_1 := \begin{bmatrix} f(aac) & f(aca)\\f(bac) & f(bca) \end{bmatrix}
\qquad\qquad
B_2 := \begin{bmatrix} g(aca) & g(aac)\\g(bca) & g(bac) \end{bmatrix}.
\]
These matrices are in $M(\alpha,\beta)$, since each can 
be realized as a coordinatewise application of $f$ or $g$ to three
generators of $M(\alpha,\beta)$.

Next define the matrix 
\[
G_1 :=\begin{bmatrix} f(baa) & f(baa)\\g(baa) & g(baa) \end{bmatrix}.
\]
Here, the idempotence of $f$ and $g$
yields that 
$f(baa)
\stackalpha f(aaa) = a = g(aaa) \stackalpha g(baa)$,
so $G_1$ is one of the
generators of $M(\alpha,\beta)$.

Next, we use the assumption that $R_{fg} \in T$ as follows.  Clearly
\[
\begin{bmatrix} b&a\\d&c\end{bmatrix}, 
\begin{bmatrix} d&c\\b&a\end{bmatrix} \in R(\beta,\alpha).
\]
Since $[\beta,\alpha]=0$ by commutativity of the commutator in $\varV$
\cite[Lemma~2.2]{kea1995}, 
applying $R_{fg} \in T$ gives
\[
\begin{bmatrix} f(baa) & g(baa)\\f(dcc) & g(dcc)\end{bmatrix},
\begin{bmatrix} f(dcc) & g(dcc)\\f(baa) & g(baa)\end{bmatrix}
\in \Mast\beta\alpha.
\]
Then taking transposes and appealing to Lemma~\ref{lm:transpose} we get that
the following are in $\Mast\alpha\beta$:
\[
H_1 := \begin{bmatrix} f(baa) &f(dcc)\\g(baa) & g(dcc)\end{bmatrix}
\qquad\qquad
H_2 := \begin{bmatrix} f(dcc) &f(baa)\\g(dcc) & g(baa)\end{bmatrix}.
\]
Since $\Mast\alpha\beta$ is a subalgebra of $\m a^{2\times 2}$ by
Lemma~\ref{lm:subalg}, we can apply $f$ coordinatewise 
to $G_1,H_1,H_2$ to get another matrix in $\Mast\alpha\beta$:
\[
B_3 := f(G_1,H_1,H_2) = \begin{bmatrix} 
  f{\circ}f \begin{pmatrix} baa\\[-1mm] baa\\[-1mm] dcc\end{pmatrix} 
& f{\circ}f \begin{pmatrix} baa\\[-1mm] dcc\\[-1mm] baa\end{pmatrix} \\
\rule{0mm}{10mm}
  f{\circ}g \begin{pmatrix} baa\\[-1mm] baa\\[-1mm] dcc\end{pmatrix} 
& f{\circ}g \begin{pmatrix} baa\\[-1mm] dcc\\[-1mm] baa\end{pmatrix} 
\end{bmatrix}
\] 
where we are using the notation 
$f{\circ}h \begin{pmatrix} \tup{x}\\[-1mm] \tup{y}\\[-1mm] \tup{z}\end{pmatrix} $ for $f(h(\tup{x}),h(\tup{y}),h(\tup{z}))$.

We now focus on the matrices $B_1,B_2,B_3 \in \Mast\alpha\beta$.  Let's
define new names for (some of) their entries as follows:
\begin{align*}
B_1 &= \begin{bmatrix} f(aac) & f(aca)\\f(bac) & f(bca) \end{bmatrix}
= \begin{bmatrix} f(aac) & f(aca)\\ r_0 & t_0\end{bmatrix}\\
B_2 &= \begin{bmatrix} g(aca) & g(aac)\\g(bca) & g(bac) \end{bmatrix}
= \begin{bmatrix} g(aca) & g(aac)\\ u_0 & g(bac)\end{bmatrix}\\
B_3 &= \begin{bmatrix} 
  f{\circ}f \begin{pmatrix} baa\\[-1mm] baa\\[-1mm] dcc\end{pmatrix} 
& f{\circ}f \begin{pmatrix} baa\\[-1mm] dcc\\[-1mm] baa\end{pmatrix} \\
\rule{0mm}{10mm}
  f{\circ}g \begin{pmatrix} baa\\[-1mm] baa\\[-1mm] dcc\end{pmatrix} 
& f{\circ}g \begin{pmatrix} baa\\[-1mm] dcc\\[-1mm] baa\end{pmatrix} 
\end{bmatrix} 
= \begin{bmatrix} r & t\\ s& u\end{bmatrix}.
\end{align*}
These matrices ``fit roughly'' together as illustrated in Figure~\ref{fig1}.
The parallel \greencol\ lines between the upper-right entry of $B_1$ and the 
upper-left entry of $B_2$ indicate an equality (in this case due to 
the assumption \eqref{4.1ass}).
Let $\theta = \alpha\cap\beta$ and note that $[\theta,\theta]\leq
[\alpha,\beta]=0$, so $\theta$ is abelian.
The red squiggly lines indicate pairs that are in $\theta$,
as shown in the next claim.

\begin{figure}
\begin{tikzpicture}

\node (faac) at (0.3,5) {$f(aac)$};
\node (r0) at (0,3) {$r_0$};
\node (r) at (0,2) {$r$};
\node (s) at (0,0) {$s$};
\node (t) at (3,2) {$t$};
\node (u) at (3,0) {$u$};
\node (t0) at (3.5,3) {$t_0$};
\node (faca) at (3.2,5) {$f(aca)$};

\node (gaca) at (5,5) {$g(aca)$};
\node (gaac) at (7.9,5) {$g(aac)$};
\node (u0) at (4.7,-.5) {$u_0$};
\node (gbac) at (7.9,-.5) {$g(bac)$};

\draw [-] (-.3,-.2) -- (3.2,-.2) -- (3.2,2.2) -- (-.3,2.2) -- cycle;
\draw [-] (4.4,-.7) -- (8.5,-.7) -- (8.5,5.3) -- (4.4,5.3) -- cycle;
\draw [-] (-.4,2.8) -- (3.8,2.8) -- (3.8,5.3) -- (-.4,5.3) -- cycle;

\node at (1.45,1) {$B_3$};
\node at (6.45,2.3) {$B_2$};
\node at (1.7,4.05) {$B_1$};

\draw [-,very thick,snake=snake,segment amplitude=.4mm,segment length=2mm,color=red] (0,2.2) -- (0,2.8);

\draw [-,very thick,snake=snake,segment amplitude=.4mm,segment length=2mm,color=red] (3,2.2) -- (3.5,2.8);

\draw [-,very thick,snake=snake,segment amplitude=.4mm,segment length=2mm,color=red] (3.2,0) -- (4.4,-.5);

\draw[-,ultra thick,color=\greencol] (3.8,5.05) -- (4.4,5.05);
\draw[-,ultra thick,color=\greencol] (3.8,4.95) -- (4.4,4.95);

\end{tikzpicture}

\caption{} \label{fig1}
\end{figure}

\begin{clm} \label{clm:theta}
$(r,r_0),(t,t_0),(u,u_0) \in \theta$.
\end{clm}

\begin{proof}
$r=f(f(baa),f(baa),f(dcc)) \stackalpha f(f(aaa),f(aaa),f(ccc))=f(aac)$ by the
idempotence of $f$,
and $f(aac) \stackalpha f(bac)=r_0$.  Thus $(r,r_0) \in \alpha$.
Similarly,
\[
r \stackbeta f(f(baa),f(baa),f(baa)) = f(baa) \stackbeta f(bac) = r_0
\]
proving $(r,r_0)\in\beta$.  So $(r,r_0) \in \alpha\cap\beta = \theta$.  The
proof is similar for $(t,t_0)$.  For $(u,u_0)$, also use the fact that
$f(aca)=g(aca)$ by the assumption~\eqref{4.1ass}.
\end{proof}

Next, we use the difference term $p$ and the fact that $(r,r_0) \in \theta$ 
and $(r_0,t_0) \in \beta$
to find $t_1$ so that
$\begin{bmatrix}r_0&t_0\\r&t_1\end{bmatrix} \in M(\theta,\beta)$.
Namely, we observe that
\[
\begin{bmatrix} r_0&r_0\\r&r\end{bmatrix},
\begin{bmatrix} r_0&r_0\\r_0&r_0\end{bmatrix},
\begin{bmatrix} r_0&t_0\\r_0&t_0\end{bmatrix} 
\]
are all generators for $M(\theta,\beta)$; applying $p$ to them coordinatewise
and noting that $p(r,r_0,r_0)=r$ as $\theta$ is abelian gives 
\[
C_1 := \begin{bmatrix}r_0&t_0\\r&t_1\end{bmatrix} \in M(\theta,\beta) \qquad\mbox{where
$t_1 = p(r,r_0,t_0)$.}
\]
In particular, we get that
$t_0=p(r_0,r_0,t_0)\stacktheta p(r,r_0,t_0)=t_1$,
so $(t_0,t_1) \in \theta$, and therefore $(t,t_1) \in \theta$
by transitivity.

Similarly, applying $p$ to
generators $\bmat tt{t_1}{t_1},\bmat {t_1}t{t_1}t,
\bmat {t_1}u{t_1}u \in G(\alpha,\theta)$
we get a matrix 
\[
C_2:= \begin{bmatrix}t&t_1\\u&u_1\end{bmatrix} \in M(\alpha,\theta)\qquad
\mbox{where $u_1 = p(t_1,t,u)$}.
\]
We get that $(u,u_1) \in \theta$ so 
$(u_0,u_1) \in \theta$. Employing the argument one more time,
applying $p$ to
generators $\bmat {u_1}{u_0}{u_1}{u_0},\bmat {u_1}{u_1}{u_1}{u_1},
\bmat ss{u_1}{u_1} \in G(\theta,\beta)$,
we get a matrix
\[
C_3:= \begin{bmatrix}s&u_1\\s_1&u_0\end{bmatrix} \in M(\theta,\beta)\qquad
\mbox{where $s_1 = p(u_0,u_1,s)$}.
\]

The matrices $C_1,C_2,C_3$ fit perfectly with $B_1,B_2,B_3$ as shown in
Figure~\ref{fig2}.  

\begin{figure}
\begin{tikzpicture}

\draw [-] (-.3,-1.7) -- (4.2,-1.7) -- (4.2,-.7) -- (-.3,-.7) -- cycle;
\draw [-] (-.3,-.2) -- (2.7,-.2) -- (2.7,1.8) -- (-.3,1.8) -- cycle;
\draw [-] (3.2,-.2) -- (4.2,-.2) -- (4.2,1.8) -- (3.2,1.8) -- cycle;
\draw [-] (-.3,2.3) -- (4.2,2.3) -- (4.2,3.3) -- (-.3,3.3) -- cycle;
\draw [-] (-.3,3.8) -- (4.2,3.8) -- (4.2,5.8) -- (-.3,5.8) -- cycle;
\draw [-] (4.7,-1.7) -- (8.2,-1.7) -- (8.2,5.8) -- (4.7,5.8) -- cycle;

\node at (1.95,-1.2) {$C_3$};
\node at (1.2,.8) {$B_3$};
\node at (3.7,.8) {$C_2$};
\node at (1.95,2.8) {$C_1$};
\node at (1.95,4.8) {$B_1$};
\node at (6.45,2.05) {$B_2$};

\node at (0,-1.5) {$s_1$};
\node at (-.08,-.9) {$s$};
\node at (4,-1.5) {$u_0$};
\node at (4,-.9) {$u_1$};
\node at (-.08,0) {$s$};
\node at (2.5,0) {$u$};
\node at (-.08,1.6) {$r$};
\node at (2.5,1.6) {$t$};
\node at (3.4,1.6) {$t$};
\node at (3.4,0) {$u$};
\node at (4,0) {$u_1$};
\node at (4,1.56) {$t_1$};
\node at (-.08,2.5) {$r$};
\node at (0,3.1) {$r_0$};
\node at (4,2.5) {$t_1$};
\node at (4,3.1) {$t_0$};
\node at (4,4) {$t_0$};
\node at (0,4) {$r_0$};
\node at (.35,5.55) {$f(aac)$};
\node at (3.55,5.55) {$f(aca)$};
\node at (5.35,5.55) {$g(aca)$};
\node at (7.55,5.55) {$g(aac)$};
\node at (7.55,-1.45) {$g(bac)$};
\node at (4.93,-1.5) {$u_0$};

\draw[-,ultra thick,color=\greencol] (4.2,-1.55) -- (4.7,-1.55);
\draw[-,ultra thick,color=\greencol] (4.2,-1.45) -- (4.7,-1.45);
\draw[-,ultra thick,color=\greencol] (4.05,-.7) -- (4.05,-.2);
\draw[-,ultra thick,color=\greencol] (3.95,-.7) -- (3.95,-.2);

\draw[-,ultra thick,color=\greencol] (.05,-.7) -- (.05,-.2);
\draw[-,ultra thick,color=\greencol] (-.05,-.7) -- (-.05,-.2);
\draw[-,ultra thick,color=\greencol] (2.7,-.05) -- (3.2,-.05);
\draw[-,ultra thick,color=\greencol] (2.7,.05) -- (3.2,.05);
\draw[-,ultra thick,color=\greencol] (2.7,1.55) -- (3.2,1.55);
\draw[-,ultra thick,color=\greencol] (2.7,1.65) -- (3.2,1.65);

\draw[-,ultra thick,color=\greencol] (.05,1.8) -- (.05,2.3);
\draw[-,ultra thick,color=\greencol] (-.05,1.8) -- (-.05,2.3);
\draw[-,ultra thick,color=\greencol] (4.05,1.8) -- (4.05,2.3);
\draw[-,ultra thick,color=\greencol] (3.95,1.8) -- (3.95,2.3);

\draw[-,ultra thick,color=\greencol] (.05,3.3) -- (.05,3.8);
\draw[-,ultra thick,color=\greencol] (-.05,3.3) -- (-.05,3.8);
\draw[-,ultra thick,color=\greencol] (4.05,3.3) -- (4.05,3.8);
\draw[-,ultra thick,color=\greencol] (3.95,3.3) -- (3.95,3.8);

\draw[-,ultra thick,color=\greencol] (4.2,5.55) -- (4.7,5.55);
\draw[-,ultra thick,color=\greencol] (4.2,5.65) -- (4.7,5.65);

\end{tikzpicture}
\caption{}\label{fig2}
\end{figure}

\begin{clm} \label{clm:top}
$\begin{bmatrix}f(aac)&g(aac)\\s_1&g(bac) \end{bmatrix} \in \Mast\alpha\beta$.
\end{clm}

\begin{proof}
Glue $B_3$ to $C_2$ horizontally; then combine the resulting matrix with
$B_1$, $C_1$ and $C_3$ vertically; and finally glue this last matrix to
$B_2$ horizontally.
\end{proof}

We now follow similar steps to find $\bar{s}_1\in A$ such that
$\begin{bmatrix}\bar{s}_1&g(bac)\\f(bbd)&g(bbd)\end{bmatrix}
\in \Mast\alpha\beta$.
The following are in $M(\alpha,\beta)$,
as can be shown by the same reasoning used to show
that $B_1$ and $B_2$ are in $M(\alpha,\beta)$:
\[
B_4 := \begin{bmatrix} f(bac) & f(dac)\\f(bbd) & f(dbd) \end{bmatrix}
\qquad\qquad
B_5 := \begin{bmatrix} g(dac) & g(bac)\\g(dbd) & g(bbd) \end{bmatrix}.
\]
The following is a generator of $M(\alpha,\beta)$,
as can be shown by the same reasoning used to show
that $G_1$ is a generator:
\[
G_2 :=\begin{bmatrix} f(dcc) & f(dcc)\\g(dcc) & g(dcc) \end{bmatrix}.
\]
Also recall the matrices $G_1,H_1 \in \Mast\alpha\beta$ defined earlier.
Now apply $f$ coordinatewise to $H_1,G_1,G_2$ to get
\[
B_6' := f(H_1,G_1,G_2) = \begin{bmatrix} 
  f{\circ}f \begin{pmatrix} baa\\[-1mm] baa\\[-1mm] dcc\end{pmatrix} 
& f{\circ}f \begin{pmatrix} dcc\\[-1mm] baa\\[-1mm] dcc\end{pmatrix} \\
\rule{0mm}{10mm}
  f{\circ}g \begin{pmatrix} baa\\[-1mm] baa\\[-1mm] dcc\end{pmatrix} 
& f{\circ}g \begin{pmatrix} dcc\\[-1mm] baa\\[-1mm] dcc\end{pmatrix} 
\end{bmatrix}.
\] 
Finally, we let $B_6$ be the matrix obtained from $B_6'$ be swapping its
top row with its bottom row.  
$\Mast\alpha\beta$ is closed under the action of swapping rows, since
$M(\alpha,\beta)$ is, so $B_6 \in \Mast\alpha\beta$.

Observe that the upper-left entry of $B_4$ equals the lower-left
entry of $B_1$, which we have named $r_0$.  Similarly, the two entries
in the first column of $B_6$ equal the two entries of the first column
of $B_3$ but in reverse order (these are $r$ and $s$).

Define a few more new names for entries of $B_4,B_5,B_6$ as follows:
\begin{align*}
B_4 &= \begin{bmatrix} f(bac) & f(dac)\\f(bbd) & f(dbd) \end{bmatrix}
= \begin{bmatrix} r_0 & \bart_0\\ f(bbd) & f(dbd)\end{bmatrix}\\
B_5 &= \begin{bmatrix} g(dac) & g(bac)\\g(dbd) & g(bbd) \end{bmatrix}
= \begin{bmatrix} \baru_0 & g(bac)\\ g(dbd) & g(bbd)\end{bmatrix}
\end{align*}

\begin{align*}
B_6 &= \begin{bmatrix} 
  f{\circ}g \begin{pmatrix} baa\\[-1mm] baa\\[-1mm] dcc\end{pmatrix} 
& f{\circ}g \begin{pmatrix} dcc\\[-1mm] baa\\[-1mm] dcc\end{pmatrix} \\
\rule{0mm}{10mm}
  f{\circ}f \begin{pmatrix} baa\\[-1mm] baa\\[-1mm] dcc\end{pmatrix} 
& f{\circ}f \begin{pmatrix} dcc\\[-1mm] baa\\[-1mm] dcc\end{pmatrix} 
\end{bmatrix} 
= \begin{bmatrix} s & \baru\\ r& \bart\end{bmatrix}.
\end{align*}

These matrices ``fit roughly'' together as illustrated in Figure~\ref{fig3}.

\begin{figure}[h]
\begin{tikzpicture}

\node (faac) at (0.3,-5) {$f(bbd)$};
\node (r0) at (0,-3) {$r_0$};
\node (r) at (0,-2) {$r$};
\node (s) at (0,0) {$s$};
\node (t) at (3,-2) {$\bart$};
\node (u) at (3,0) {$\baru$};
\node (t0) at (3.5,-3.05) {$\bart_0$};
\node (faca) at (3.2,-5) {$f(dbd)$};

\node (gaca) at (5,-5) {$g(dbd)$};
\node (gaac) at (7.9,-5) {$g(bbd)$};
\node (u0) at (4.7,.5) {$\baru_0$};
\node (gbac) at (7.9,.5) {$g(bac)$};

\draw [-] (-.3,.2) -- (3.2,.2) -- (3.2,-2.2) -- (-.3,-2.2) -- cycle;
\draw [-] (4.4,.7) -- (8.5,.7) -- (8.5,-5.3) -- (4.4,-5.3) -- cycle;
\draw [-] (-.4,-2.8) -- (3.8,-2.8) -- (3.8,-5.3) -- (-.4,-5.3) -- cycle;

\node at (1.45,-1) {$B_6$};
\node at (6.45,-2.3) {$B_5$};
\node at (1.7,-4.05) {$B_4$};

\draw [-,very thick,snake=snake,segment amplitude=.4mm,segment length=2mm,color=red] (0,-2.2) -- (0,-2.8);

\draw [-,very thick,snake=snake,segment amplitude=.4mm,segment length=2mm,color=red] (3,-2.2) -- (3.5,-2.8);

\draw [-,very thick,snake=snake,segment amplitude=.4mm,segment length=2mm,color=red] (3.2,0) -- (4.4,.5);

\draw[-,ultra thick,color=\greencol] (3.8,-5.05) -- (4.4,-5.05);
\draw[-,ultra thick,color=\greencol] (3.8,-4.95) -- (4.4,-4.95);

\end{tikzpicture}

\caption{} \label{fig3}
\end{figure}

\begin{clm} \label{clm:thetabar}
$(\bart,\bart_0),(\baru,\baru_0) \in \theta$.
\end{clm}

\begin{proof}
Similar to the proof of Claim~\ref{clm:theta}.
\end{proof}

Now we repeat the process of adding matrices from $M(\theta,\beta)$
or $M(\alpha,\theta)$.  Because $(r,r_0) \in \theta$ (by Claim~\ref{clm:theta})
we can apply $p$ to 
$\bmat r{r_0}r{r_0}$, $\bmat {r_0}{r_0}{r_0}{r_0},
\bmat {r_0}{r_0}{\bart_0}{\bart_0} \in G(\theta,\beta)$ to
get
\[
C_4 := \begin{bmatrix}r&\bart_1\\r_0&\bart_0\end{bmatrix} \in M(\theta,\beta) \qquad\mbox{where
$\bart_1 = p(r,r_0,\bart_0)$.}
\]
In particular, we get that $(\bart_0,\bart_1) \in \theta$ so $(\bart,\bart_1) 
\in \theta$ by transitivity.

By a similar argument, applying $p$ to
$\bmat \bart\bart{\bart_1}{\bart_1},\bmat\bart{\bart_1}\bart{\bart_1},
\bmat\baru{\bart_1}\baru{\bart_1} \in G(\alpha,\theta)$
we get a matrix 
\[
C_5:= \begin{bmatrix}\baru&\baru_1\\ \bart&\bart_1\end{bmatrix} \in M(\alpha,\theta)\qquad
\mbox{where $\baru_1 = p(\bart_1,\bart,\baru)$}.
\]
We get $(\baru_0,\baru_1) \in \theta$,
so applying $p$ to
$\bmat {\baru_0}{\baru_1}{\baru_0}{\baru_1},\bmat {\baru_1}{\baru_1}
{\baru_1}{\baru_1},\bmat ss{\baru_1}{\baru_1} \in G(\theta,\beta)$
we get 
\[
C_6:= \begin{bmatrix}\bars_1&\baru_0\\s&\baru_1\end{bmatrix} \in M(\theta,\beta)\qquad
\mbox{where $\bars_1 = p(\baru_0,\baru_1,s)$}.
\]

\begin{figure}
\begin{tikzpicture}

\draw [-] (-.3,1.7) -- (4.2,1.7) -- (4.2,.7) -- (-.3,.7) -- cycle;
\draw [-] (-.3,.2) -- (2.7,.2) -- (2.7,-1.8) -- (-.3,-1.8) -- cycle;
\draw [-] (3.2,.2) -- (4.2,.2) -- (4.2,-1.8) -- (3.2,-1.8) -- cycle;
\draw [-] (-.3,-2.3) -- (4.2,-2.3) -- (4.2,-3.3) -- (-.3,-3.3) -- cycle;
\draw [-] (-.3,-3.8) -- (4.2,-3.8) -- (4.2,-5.8) -- (-.3,-5.8) -- cycle;
\draw [-] (4.7,1.7) -- (8.2,1.7) -- (8.2,-5.8) -- (4.7,-5.8) -- cycle;

\node at (1.95,1.2) {$C_6$};
\node at (1.2,-.8) {$B_6$};
\node at (3.7,-.8) {$C_5$};
\node at (1.95,-2.8) {$C_4$};
\node at (1.95,-4.8) {$B_4$};
\node at (6.45,-2.05) {$B_5$};

\node at (0,1.5) {$\bars_1$};
\node at (-.08,.9) {$s$};
\node at (4,1.5) {$\baru_0$};
\node at (4,.9) {$\baru_1$};
\node at (-.08,0) {$s$};
\node at (2.5,0) {$\baru$};
\node at (-.08,-1.6) {$r$};
\node at (2.5,-1.55) {$\bart$};
\node at (3.4,-1.55) {$\bart$};
\node at (3.4,0) {$\baru$};
\node at (4,0) {$\baru_1$};
\node at (4,-1.56) {$\bart_1$};
\node at (-.08,-2.5) {$r$};
\node at (0,-3.1) {$r_0$};
\node at (4,-2.55) {$\bart_1$};
\node at (4,-3.1) {$\bart_0$};
\node at (4,-4.05) {$\bart_0$};
\node at (0,-4) {$r_0$};
\node at (.35,-5.55) {$f(bbd)$};
\node at (3.55,-5.55) {$f(dbd)$};
\node at (5.35,-5.55) {$g(dbd)$};
\node at (7.55,-5.55) {$g(bbd)$};
\node at (7.55,1.45) {$g(bac)$};
\node at (4.93,1.5) {$\baru_0$};

\draw[-,ultra thick,color=\greencol] (4.2,1.55) -- (4.7,1.55);
\draw[-,ultra thick,color=\greencol] (4.2,1.45) -- (4.7,1.45);
\draw[-,ultra thick,color=\greencol] (4.05,.7) -- (4.05,.2);
\draw[-,ultra thick,color=\greencol] (3.95,.7) -- (3.95,.2);

\draw[-,ultra thick,color=\greencol] (.05,.7) -- (.05,.2);
\draw[-,ultra thick,color=\greencol] (-.05,.7) -- (-.05,.2);
\draw[-,ultra thick,color=\greencol] (2.7,.05) -- (3.2,.05);
\draw[-,ultra thick,color=\greencol] (2.7,-.05) -- (3.2,-.05);
\draw[-,ultra thick,color=\greencol] (2.7,-1.55) -- (3.2,-1.55);
\draw[-,ultra thick,color=\greencol] (2.7,-1.65) -- (3.2,-1.65);

\draw[-,ultra thick,color=\greencol] (.05,-1.8) -- (.05,-2.3);
\draw[-,ultra thick,color=\greencol] (-.05,-1.8) -- (-.05,-2.3);
\draw[-,ultra thick,color=\greencol] (4.05,-1.8) -- (4.05,-2.3);
\draw[-,ultra thick,color=\greencol] (3.95,-1.8) -- (3.95,-2.3);

\draw[-,ultra thick,color=\greencol] (.05,-3.3) -- (.05,-3.8);
\draw[-,ultra thick,color=\greencol] (-.05,-3.3) -- (-.05,-3.8);
\draw[-,ultra thick,color=\greencol] (4.05,-3.3) -- (4.05,-3.8);
\draw[-,ultra thick,color=\greencol] (3.95,-3.3) -- (3.95,-3.8);

\draw[-,ultra thick,color=\greencol] (4.2,-5.55) -- (4.7,-5.55);
\draw[-,ultra thick,color=\greencol] (4.2,-5.65) -- (4.7,-5.65);

\end{tikzpicture}
\caption{}\label{fig4}
\end{figure}

The matrices $C_4,C_5,C_6$ fit perfectly with $B_4,B_5,B_6$ as shown in
Figure~\ref{fig4}.  This proves

\begin{clm} \label{clm:bot}
$\begin{bmatrix} \bars_1 & g(bac)\\ f(bbd) & g(bbd) \end{bmatrix}
\in \Mast\alpha\beta$.
\end{clm}

Now for our final bit of magic, we will show

\begin{clm} \label{clm:sequal}
$s_1=\bars_1$.
\end{clm}

\begin{proof}
First, we define four more matrices in $M(\alpha,\beta)$:
\begin{align*}
B_7 &= \begin{bmatrix} f(bca) & f(baa)\\f(bcb) & f(bab) \end{bmatrix}
&B_9 = \begin{bmatrix} f(dcc) & f(dac)\\f(dcd) & f(dad) \end{bmatrix}
\phantom{.}\\
B_8 &= \rule{0mm}{6mm}\begin{bmatrix} g(bcb) & g(bab)\\g(bca) & g(baa) \end{bmatrix}
&B_{10} = \begin{bmatrix} g(dcd) & g(dad)\\g(dcc) & g(dac) \end{bmatrix}.
\end{align*}
With the matrix $H_1$ defined earlier, these matrices fit perfectly
as shown in Figure~\ref{fig5}.

\begin{figure}
\begin{tikzpicture}

\draw[-] (-.3,-.2) -- (2.7,-.2) -- (2.7,1.8) -- (-.3,1.8) -- cycle;
\draw[-] (-.3,2.3) -- (2.7,2.3) -- (2.7,4.3) -- (-.3,4.3) -- cycle;
\draw[-] (6.7,-.2) -- (9.7,-.2) -- (9.7,1.8) -- (6.7,1.8) -- cycle;
\draw[-] (6.7,2.3) -- (9.7,2.3) -- (9.7,4.3) -- (6.7,4.3) -- cycle;
\draw[-] (3.2,-.2) -- (6.2,-.2) -- (6.2,4.3) -- (3.2,4.3) -- cycle;

\node at (1.2,.8) {$B_8$};
\node at (1.2,3.3) {$B_7$};
\node at (8.2,.8) {$B_{10}$};
\node at (8.2,3.3) {$B_9$};
\node at (4.7,2.05) {$H_1$};

\node at (.3,4.0) {$f(bca)$};
\node at (2.1,4.0) {$f(baa)$};
\node at (.3,2.6) {$f(bcb)$};
\node at (2.1,2.6) {$f(bab)$};

\node at (.3,1.5) {$g(bcb)$};
\node at (2.1,1.5) {$g(bab)$};
\node at (.3,0.1) {$g(bca)$};
\node at (2.1,0.1) {$g(baa)$};

\node at (.3,4.0) {$f(bca)$};
\node at (2.1,4.0) {$f(baa)$};
\node at (.3,2.6) {$f(bcb)$};
\node at (2.1,2.6) {$f(bab)$};

\node at (7.3,1.5) {$g(dcd)$};
\node at (9.1,1.5) {$g(dad)$};
\node at (7.3,0.1) {$g(dcc)$};
\node at (9.1,0.1) {$g(dac)$};

\node at (7.3,4.0) {$f(dcc)$};
\node at (9.1,4.0) {$f(dac)$};
\node at (7.3,2.6) {$f(dcd)$};
\node at (9.1,2.6) {$f(dad)$};

\node at (3.8,0.1) {$g(baa)$};
\node at (5.6,0.1) {$g(dcc)$};
\node at (3.8,4.0) {$f(baa)$};
\node at (5.6,4.0) {$f(dcc)$};

\draw[-,ultra thick,color=\greencol] (2.7,.05) -- (3.2,.05);
\draw[-,ultra thick,color=\greencol] (2.7,.15) -- (3.2,.15);
\draw[-,ultra thick,color=\greencol] (2.7,4.05) -- (3.2,4.05);
\draw[-,ultra thick,color=\greencol] (2.7,3.95) -- (3.2,3.95);

\draw[-,ultra thick,color=\greencol] (6.7,.05) -- (6.2,.05);
\draw[-,ultra thick,color=\greencol] (6.7,.15) -- (6.2,.15);
\draw[-,ultra thick,color=\greencol] (6.7,4.05) -- (6.2,4.05);
\draw[-,ultra thick,color=\greencol] (6.7,3.95) -- (6.2,3.95);

\draw[-,ultra thick,color=\greencol] (0.25,1.8) -- (0.25,2.3);
\draw[-,ultra thick,color=\greencol] (0.35,1.8) -- (0.35,2.3);
\draw[-,ultra thick,color=\greencol] (2.05,1.8) -- (2.05,2.3);
\draw[-,ultra thick,color=\greencol] (2.15,1.8) -- (2.15,2.3);

\draw[-,ultra thick,color=\greencol] (7.25,1.8) -- (7.25,2.3);
\draw[-,ultra thick,color=\greencol] (7.35,1.8) -- (7.35,2.3);
\draw[-,ultra thick,color=\greencol] (9.05,1.8) -- (9.05,2.3);
\draw[-,ultra thick,color=\greencol] (9.15,1.8) -- (9.15,2.3);

\end{tikzpicture}
\caption{}\label{fig5}
\end{figure}

Thus we get 
\begin{equation} \label{eq:1}
\begin{bmatrix} t_0 & \bart_0\\u_0 & \baru_0\end{bmatrix}=
\begin{bmatrix} f(bca) & f(dac)\\
g(bca) &  g(dac) \end{bmatrix} \in \Mast\alpha\beta.
\end{equation}

Recall that the following matrices are in $\Mast\alpha\beta$:
\[
B_3 = \begin{bmatrix} r&t\\s&u\end{bmatrix}
\qquad\mbox{and}\qquad
B_6' = \begin{bmatrix} r & \bart\\s & \baru\end{bmatrix}.
\]
If we swap the two columns of $B_3$, we can horizontally glue the resulting 
matrix to $B_6'$, getting
\begin{equation} \label{eq:2}
\begin{bmatrix} t & \bart\\u & \baru\end{bmatrix} \in \Mast\alpha\beta.
\end{equation}

It follows from \eqref{eq:1} and \eqref{eq:2} that we can apply
$p(x_1,p(x_2,x_3,x_4),x_5)$ to the matrices
\[
\begin{bmatrix} u_0&\baru_0\\u_0&\baru_0 \end{bmatrix},\quad
\begin{bmatrix} t_0&\bart_0\\u_0&\baru_0 \end{bmatrix},\quad
\begin{bmatrix} t&\bart\\u&\baru \end{bmatrix},\quad
\begin{bmatrix} u&\baru\\u&\baru \end{bmatrix},\quad
\begin{bmatrix} s&s\\s&s \end{bmatrix}
\]
and the resulting matrix 
\[
\begin{bmatrix}
p(u_0,p(t_0,t,u),s) & \rule{2mm}{0mm} & p(\baru_0,p(\bart_0,\bart,\baru),s)\\
p(u_0,p(u_0,u,u),s) && p(\baru_0,p(\baru_0,\baru,\baru),s)
\end{bmatrix}
\]
is in $\Mast\alpha\beta$.  Observe that $(u,u_0),(\baru,\baru_0)\in\theta$
by Claims~\ref{clm:theta} and~\ref{clm:thetabar}.  Hence the bottom entries
of the above matrix are both equal to $s$.  By Lemma~\ref{lm:C2*}(1) we get
that the top entries are also equal, i.e.,
\begin{equation}
p(u_0,p(t_0,t,u),s) = p(\baru_0,p(\bart_0,\bart,\baru),s).
\end{equation}
Let's rewrite this last equation as
\begin{equation} \label{eq:4}
p(u_0,p(\underbrace{p(r,\underline{r},t_0)}_{=\,t_0},t,u),s) = 
p(\baru_0,p(\underbrace{p(r,\underline{r},\bart_0)}_{=\,\bart_0},\bart,\baru),s).
\end{equation}
Observe that
\[
t \stackbeta r \stackbeta \bart \qquad\mbox{and}\qquad
u \stackbeta s \stackbeta \baru.
\]
Since $(t_0,t),(\bart,\bart_0),(u_0,u),(\baru,\baru_0) \in \theta$, we get
that
each of the pairs $(u_0,\baru_0)$, $(t_0,\bart_0)$, $(t,\bart), (u,\baru)$ is
in $\beta$.  As $(r,r_0) \in \theta\subseteq\alpha$, the condition
$[\alpha,\beta]=0$ allows us to change the two underlined occurrences of $r$
in equation~\eqref{eq:4} to $r_0$, producing
\begin{equation}
p(u_0,p(p(r,r_0,t_0),t,u),s) = 
p(\baru_0,p(p(r,r_0,\bart_0),\bart,\baru),s).
\end{equation}
Now
\begin{align*}
p(u_0,p(p(r,r_0,t_0),t,u),s) &=  p(u_0,p(t_1,t,u),s) &\mbox{definition of $t_1$}\\
&= p(u_0,u_1,s) &\mbox{definition of $u_1$}\\
&= s_1 &\mbox{definition of $s_1$}
\end{align*}
and 
$p(\baru_0,p(p(r,r_0,\bart_0),\bart,\baru),s)$ similarly simplifies to
$\bars_1$.  Thus $s_1=\bars_1$, proving Claim~\ref{clm:sequal}.
\end{proof}

Now we can finish the proof of Theorem~\ref{thm:laststep}.  
By Claims~\ref{clm:top}, \ref{clm:bot} and \ref{clm:sequal}, we have
matrices
\[
\begin{bmatrix}f(aac) & g(aac)\\s_1 & g(bac)\end{bmatrix},
\begin{bmatrix}s_1 & g(bac)\\f(bbd) & g(bbd)\end{bmatrix} \in 
\Mast\alpha\beta.
\]
We can glue the first matrix on top of the second to get
\[
\begin{bmatrix} f(aac) & g(aac)\\f(bbd) & g(bbd) \end{bmatrix} 
\in \Mast\alpha \beta.
\]
As $\m a,\alpha,\beta,\begin{bmatrix}a&c\\b&d\end{bmatrix}$ were arbitrary,
this proves
\[
L_{fg} = \begin{bmatrix} f(xxz) & g(xxz)\\f(yyw) & g(yyw) \end{bmatrix} \in T.
\qedhere
\]
\end{proof}

\section{Proof of Theorem~\ref{thm:main}} \label{section:assembly}

\begin{proof}[Proof of Theorem~\ref{thm:main}]
Given that $\m a \in \varV$ and $[\alpha,\beta]=0$, we must show that 
$q$ restricted to $R(\alpha,\beta)$ is a homomorphism
$R(\alpha,\beta)\ra \m a$.
Let $s$ be an $n$-ary term and $\begin{bmatrix}a_i&c_i\\b_i&d_i\end{bmatrix}
\in R(\alpha,\beta)$ for $i=1,\ldots,n$.  Applying $s$ gives
\[
\begin{bmatrix}s(\tup{a}) & s(\tup{c})\\s(\tup{b}) & s(\tup{d})
\end{bmatrix} \in R(\alpha,\beta).
\]
For $i=1,\ldots,n$ define $c_i' = q(a_i,b_i,c_i,d_i)$.
Applying 
Lemma~\ref{lm:crucial} to the $n+1$ matrices
in $R(\alpha,\beta)$ at hand gives

\begin{align}
&\begin{bmatrix}a_i&c_i'\\b_i&d_i\end{bmatrix} \in \Mast\alpha\beta,
\quad(i=1,\ldots,n) \label{eq1}\\
&\begin{bmatrix}s(\tup{a})&
q(s(\tup{a}),s(\tup{b}),s(\tup{c}),s(\tup{d}))\\s(\tup{b}) & s(\tup{d})
\end{bmatrix} \in \Mast\alpha\beta\label{eq2}.
\end{align}

Applying $s$ to the $n$ matrices in \eqref{eq1} and using 
Lemma~\ref{lm:subalg}, we get
\[
\begin{bmatrix} s(\tup{a})&
s(c_1',\ldots,c_n')\\s(\tup{b})&s(\tup{d})\end{bmatrix} \in \Mast\alpha\beta,
\]
which with \eqref{eq2} and Lemma~\ref{lm:C2*}(2) gives
\[
s(c_1',\ldots,c_n') = q(s(\tup{a}),s(\tup{b}),s(\tup{c}),s(\tup{d}))
\]
as required.
\end{proof}

\section{Proof of Lemma~6.2 of \cite{ksw2016}} \label{sec:Lm6.2}

In this section 
we prove Lemma~6.2 of \cite{ksw2016}.
First, we need the following fact about arbitrary Kiss terms.
(Arbitrary Kiss terms were defined by (I)${}_q$ and (II)${}_q$
of the Introduction.)

\begin{lm} \label{lm:indep}
Let $\varV$ be a variety with a Kiss term $q(x,y,z,w)$.
If $\m a\in\varV, \alpha, \beta\in\Con(\m a)$,
and $[\alpha,\beta]=0$,
then
$q(x,y,z,w)$ is independent of its third variable
on $R(\alpha,\beta)$.
\end{lm}

\begin{proof}
It follows from (II)${}_q$ of the definition of a Kiss term
that if
$(a,b,c,d)$, $(a,b,c',d)\in R(\alpha,\beta)$, then
\[
q(a,b,c,d) \stackrel{[\alpha,\beta]}{\equiv} q(a,b,c',d).
\]
Since we are assuming $[\alpha,\beta]=0$,
we obtain that
$q(a,b,c,d)=q(a,b,c',d)$
whenever $(a,b,c,d), (a,b,c',d)\in R(\alpha,\beta)$,
which is the claim of the lemma.
\end{proof}

Now 
let $\varV$ be a variety having a Kiss term $q$ which has been 
constructed from
a 3-ary difference term via Lipparini's Formula.  In this context,
Lemma~6.2 of \cite{ksw2016} states the following.

\medskip\noindent
\textbf{Lemma 6.2 of \cite{ksw2016}.} 
\emph{If $\m a \in \varV$ and $\alpha,\beta \in \Con(\m a)$, then
$[\alpha,\beta]=0$ iff 
\begin{enumerate}
\item[(i)]
$q\colon R(\alpha,\beta)\ra \m a$ is a homomorphism, and
\item[(ii)]
$q$ is independent of its third variable on $R(\alpha,\beta)$.
\end{enumerate}
}

\medskip
In the following proof, note 
that the requirement that $q$ be constructed from a 3-ary difference term 
via Lipparini's Formula is used only at the point where 
Theorem~\ref{thm:main} is called.

\begin{proof}[Proof of Lemma~6.2 of \cite{ksw2016}]
Assume that $[\alpha,\beta]=0$.  
Item (ii) holds by Lemma~\ref{lm:indep}.
Item (i) holds by Theorem~\ref{thm:main}, using our assumption on $q$.

Conversely, assume that Items (i) and (ii) hold for
some Kiss term $q$. Recall that $M(\alpha,\beta)$
is the subalgebra of $R(\alpha,\beta)$ generated by
\[
G(\alpha,\beta) :=
\left\{ \begin{bmatrix} c&c\\d&d\end{bmatrix}
: (c,d) \in \alpha\right\}
\cup
\left\{
\begin{bmatrix} a&c\\a&c\end{bmatrix}
: (a,c) \in \beta\right\}.
\]
It follows from the Kiss identities (I)${}_q$
that $q(a,b,c,d)=c$ if
$\begin{bmatrix} a&c\\b&d\end{bmatrix}\in G(\alpha,\beta)$.
Assuming Item (i) of the lemma statement,  
$q\colon R(\alpha,\beta)\to \m a$ is a homomorphism.
This homomorphism agrees with the third projection homomorphism
$\pi_3(a,b,c,d)=c$
on $G(\alpha,\beta)$.
Since $q=\pi_3$ on $G(\alpha,\beta)$,
and both $q$ and $\pi_3$ are homomorphisms,
we have $q=\pi_3$ on the generated subalgebra
$\langle G(\alpha,\beta)\rangle=M(\alpha,\beta)$.
To repeat this statement more directly, 
$q(a,b,c,d)=c$ whenever
$\begin{bmatrix} a&c\\b&d\end{bmatrix}\in M(\alpha,\beta)$.

We claim that $b=d$ implies $a=c$ whenever
$\begin{bmatrix} a&c\\b&d\end{bmatrix}\in M(\alpha,\beta)$.
For, if $b=d$ for some
$\begin{bmatrix} a&c\\b&d\end{bmatrix}\in M(\alpha,\beta)$,
  then
$\begin{bmatrix} a&c\\d&d\end{bmatrix}, \begin{bmatrix} a&a\\d&d\end{bmatrix}
\in M(\alpha,\beta)\leq R(\alpha,\beta)$.
Now we use Item (ii) of the lemma statement together with
identities (I)${}_q$ from the definition of a Kiss term to conclude that
$q(a,d,c,d)=q(a,d,a,d)=a$. Yet
$q(a,d,c,d)=c$ by the conclusion of the preceding paragraph.
This shows that $a=c$, as claimed.

We have shown that $b=d$ implies $a=c$ whenever
$\begin{bmatrix} a&c\\b&d\end{bmatrix}\in M(\alpha,\beta)$,
  which is one of the standard
  definitions of the expression ``$[\alpha,\beta]=0$''.
\end{proof}

\section{Refinements and extensions}
\label{section:refinements}

We have shown that, given a 3-ary difference term $p(x,y,z)$ for a variety,
the 4-ary Kiss term 
\[
p(p(x,z,z),p(y,w,z),z)
\]
obtained from $p$ via Lipparini's Formula satisfies certain 
properties (see our Theorem~\ref{thm:main} and  Lemma~6.2 of \cite{ksw2016}).
In this section we will see that
Theorem~\ref{thm:main} and  Lemma~6.2 of \cite{ksw2016}
hold for arbitrary Kiss terms, and not just those
defined by Lipparini's Formula.
This will be accomplished by showing that in a variety with a difference term,
if $[\alpha,\beta]=0$ then all Kiss terms for the variety agree on 
$R(\alpha,\beta)$.  
When $\alpha\leq \beta$, we refine this last result and connect it to
3-ary difference terms for the variety.
Next, we extend results in earlier sections by explaining the effect of
eliminating the hypothesis ``$[\alpha,\beta]=0$''.
Finally, we show that a result of Moorhead for
pairs of comparable congruences in 
congruence meet-semidistributive varieties is true for
arbitrary pairs of congruences.

\begin{thm} \label{thm:indep}
If $q$ and $q'$ are Kiss terms
for $\varV$, $\alpha, \beta\in\Con(\m a), \m a\in\varV$,
and $[\alpha,\beta]=0$, 
then $q$ and $q'$ agree on $R(\alpha,\beta)$.
\end{thm}

\begin{proof}
To show that any two Kiss terms agree on $R(\alpha,\beta)$
it suffices to choose 
a $3$-ary difference term $p$,
define the term $q$ from $p$
using Lipparini's Formula \rm{(\ref{LipForm})},
and then prove that any other Kiss term $q'$
agrees with this Lipparini-type
Kiss term $q$ on $R(\alpha,\beta)$.
Therefore, choose a 
$3$-ary difference term $p$
and let $q$ be given by Lipparini's Formula
applied to $p$.
Thus it follows from Theorem~\ref{thm:main}
that $q\colon R(\alpha,\beta)\to \m a$ is a homomorphism.
Let $q'$ be any other Kiss term.

Since $q$ is a homomorphism 
on $R(\alpha,\beta)$ and $q'$ is a term,
$q$ and $q'$ must commute on arrays of the form
\[
\begin{bmatrix}
a &b &a &b \\
b &b &b &b \\
a &b &c &d \\
b &b &d &d \\
\end{bmatrix}
\]
for all $(a,b,c,d)\in R(\alpha,\beta)$.
Here, all of the columns and rows
of this $4\times 4$ array belong to $R(\alpha,\beta)$.
The commutativity just claimed is the second equality in 
\begin{align*}
q(a,b,q'(a,b,c,d),d)&=
q(q'(a,b,a,b),q'(b,b,b,b),q'(a,b,c,d),q'(b,b,d,d))\\
&=
q'(q(a,b,a,b),q(b,b,b,b),q(a,b,c,d),q(b,b,d,d))\\
&=q'(a,b,q(a,b,c,d),d).
\end{align*}
The first and third equalities follow from the Kiss
identities (I)${}_q$.

Observe that for any Kiss term $q'$ we have 
$q'(a,b,c,d) \stackrel{\beta}{\equiv} q'(a,b,a,b) = a$
and
$q'(a,b,c,d) \stackrel{\alpha}{\equiv} q'(b,b,d,d) = d$.
This implies that if $(a,b,c,d)\in R(\alpha,\beta)$,
then $(a,b,q'(a,b,c,d),d)\in R(\alpha,\beta)$.
By Lemma~\ref{lm:indep} we obtain
\[
q(a,b,q'(a,b,c,d),d) = q(a,b,c,d).
\]
A similar argument establishes that
$q'(a,b,q(a,b,c,d),d) = q'(a,b,c,d)$.
Combining this with the final displayed lines
of the previous
paragraph we obtain
\[
q(a,b,c,d) =
q(a,b,q'(a,b,c,d),d) = 
q'(a,b,q(a,b,c,d),d) = 
q'(a,b,c,d),
\]
showing that $q$ and $q'$ agree on $R(\alpha,\beta)$.
\end{proof}

\begin{cor} \label{cor:indep} 
  The conclusion of Theorem~\ref{thm:main} holds
  for arbitrary Kiss terms.
\end{cor}

\begin{proof}
Theorem~\ref{thm:main} proves that if $[\alpha,\beta]=0$,
then $q\colon R(\alpha,\beta)\to \m a$ is a homomorphism
when $q$ is the Kiss term constructed 
from a 3-ary difference term
by Lipparini's Formula.
Theorem~\ref{thm:indep} proves that if $q'$
is any other Kiss term for the same variety
and $[\alpha,\beta]=0$, then
$q=q'$ on $R(\alpha,\beta)$. Thus
$q'\colon R(\alpha,\beta)\to \m a$ is also a homomorphism (the same one).
\end{proof}

\begin{cor} \label{cor:indep2}
Lemma~6.2 of \cite{ksw2016} holds for arbitrary Kiss terms.
\end{cor}

\begin{proof}
This can be deduced from Lemma~6.2 of \cite{ksw2016} 
(proved in Section~\ref{sec:Lm6.2})
using Theorem~\ref{thm:indep}, or by 
replacing the use of Theorem~\ref{thm:main} in our proof of Lemma~6.2
of~\cite{ksw2016} in Section~\ref{sec:Lm6.2} with Corollary~\ref{cor:indep}.
\end{proof}

Let $\varV$ be a variety with a difference term $p$,
let $q$ be an arbitrary Kiss term for $\varV$, and let $\m a\in\varV$.
Assume $\alpha,\beta\in\Con(\m a)$
satisfy $[\alpha,\beta]=0$.
We know from Lemma~\ref{lm:indep} that $q$ is independent of its third variable
on $R(\alpha,\beta)$.
The third variable corresponds to the top
right entry in the matrices in $R(\alpha,\beta)$.
Therefore, the restriction of $q$ to $R(\alpha,\beta)$ yields
a ternary function $q^-$ with codomain $A$ and domain
\begin{equation*}
  R^-(\alpha,\beta):=\left\{
  \begin{bmatrix}
    a & \\
    b & d
  \end{bmatrix}\in A^3 :
  \begin{bmatrix}
    a & c\\
    b & d
  \end{bmatrix}\in R(\alpha,\beta)
  \text{ for some } c\in A
  \right\}
\end{equation*}
such that $q^-(a,b,d)=q(a,b,c,d)$ for all
$\displaystyle
  \begin{bmatrix}
    a & c\\
    b & d
  \end{bmatrix}\in R(\alpha,\beta).$

\begin{thm} \label{thm:Delta-q-p}
Let $\varV$ be a variety with a difference term $p$,
let $q$ be an arbitrary Kiss term for $\varV$, and let $\m a\in\varV$.
Assume $\alpha,\beta\in\Con(\m a)$
satisfy $[\alpha,\beta]=0$. The function
$q^-$, defined above, has the following properties:
\begin{enumerate}
\item[{\rm(1)}]
  The relation $\Mast\alpha\beta$ is the graph of $q^-$; that is,
  \begin{equation}\label{eq:q-and-delta}
    \Mast\alpha\beta=\left\{
        \begin{bmatrix}
    a & q^-(a,b,d)\\
    b & d
    \end{bmatrix}:
    \begin{bmatrix}
    a & \\
    b & d
    \end{bmatrix}\in R^-(\alpha,\beta)
    \right\}.
  \end{equation}.
\item[{\rm(2)}]
  If $\alpha\le\beta$, then $q^-$ agrees with $p$ on $R^-(\alpha,\beta)$,
  hence 
  the relation $\Mast\alpha\beta$ is also the graph of $p$ restricted
  to $R^-(\alpha,\beta)$:
  \begin{equation}\label{eq:p-and-delta}
    \Mast\alpha\beta=\left\{
        \begin{bmatrix}
    a & p(a,b,d)\\
    b & d
    \end{bmatrix}:
    \begin{bmatrix}
    a & \\
    b & d
    \end{bmatrix}\in R^-(\alpha,\beta)
    \right\}.
  \end{equation}.
\end{enumerate}  
\end{thm}

\begin{proof}
  The definition of $q^-$ involves only how $q$ restricts to $R(\alpha,\beta)$.
  Therefore, by Theorem~\ref{thm:indep}, no generality is lost by
  assuming that $q$ is the special Kiss term obtained from $p$ via
  Lipparini's Formula \eqref{LipForm}, i.e., $q$ is the Kiss term
  that we used throughout Sections~2--6.

  For (1), the definition of $q^-$ shows that the right-hand side of
  \eqref{eq:q-and-delta} is equal to
  \begin{equation}\label{eq:new-rhs}
    \left\{
    \begin{bmatrix}
    a & q(a,b,c,d)\\
    b & d
    \end{bmatrix}:
    \begin{bmatrix}
    a & c\\
    b & d
    \end{bmatrix}\in R(\alpha,\beta)
    \right\}.
  \end{equation}
Therefore, the inclusion $\supseteq$ in \eqref{eq:q-and-delta}
follows from Lemma~\ref{lm:crucial}.

Recall from the definition of $\Mast\alpha\beta$
(Definition~\ref{df:delta})
that $\Mast\alpha\beta\subseteq R(\alpha,\beta)$.
Combining this with the inclusion established in the preceding paragraph,
we get that the relations on both sides of
\eqref{eq:q-and-delta} project onto $R^-(\alpha,\beta)$ when omitting the top
right entries. In the relation on the right-hand side, the top right entry is
clearly a function of the remaining three entries. The same holds for
$\Mast\alpha\beta$ by Lemma~\ref{lm:C2*}~(2).
Therefore, both sides of \eqref{eq:q-and-delta} are graphs of functions
with domain $R^-(\alpha,\beta)$.
Since one is a subset of the other, they must be equal.
This completes the proof of statement~(1).

To prove (2), assume that $\alpha\le\beta$. Our hypothesis
$[\alpha,\beta]=0$ implies that $[\alpha,\alpha]=0$.

First we establish the inclusion $\supseteq$ in \eqref{eq:p-and-delta}.
If
$\begin{bmatrix}
    a & \\
    b & d
\end{bmatrix}\in R^-(\alpha,\beta)$, then
$a\stackrel{\alpha}{\equiv}b$, so
$[\alpha,\alpha]=0$ implies that $p(a,b,b)=a$.
Hence, by applying $p$ to the following generators
\begin{equation*}
\begin{bmatrix}
    a & a\\
    b & b
\end{bmatrix},\quad
\begin{bmatrix}
    b & b\\
    b & b
\end{bmatrix},\quad
\begin{bmatrix}
    b & d\\
    b & d
\end{bmatrix}
\end{equation*}  
of $M(\alpha,\beta)$ 
we get that
\begin{align*}
\begin{bmatrix}
    a & p(a,b,d)\\
    b & d
\end{bmatrix} &{}=
\begin{bmatrix}
    p(a,b,b) & p(a,b,d)\\
    p(b,b,b) & p(b,b,d)
\end{bmatrix}\\ &{}=
  p\left(
\begin{bmatrix}
    a & a\\
    b & b
\end{bmatrix},
\begin{bmatrix}
    b & b\\
    b & b
\end{bmatrix},
\begin{bmatrix}
    b & d\\
    b & d
\end{bmatrix}
  \right)
\in M(\alpha,\beta)\subseteq\Mast\alpha\beta.
\end{align*}
This proves that $\supseteq$ holds in \eqref{eq:p-and-delta}.

Now the same argument as in the second paragraph of the proof of (1) yields
that equality holds in \eqref{eq:p-and-delta}.
Combining this with \eqref{eq:q-and-delta} we conclude that $p$ and $q^-$
agree on $R^-(\alpha,\beta)$, as claimed.
\end{proof}

Equality~\eqref{eq:q-and-delta} yields a characterization
of $\Mast\alpha\beta$ in the situation where
$\varV$ is a variety with a 
Kiss term $q$, $\m a\in\varV$, and
$\alpha,\beta\in\Con(\m a)$ satisfy
$[\alpha,\beta]=0$.
In the following corollary we reformulate this characterization in
two different ways, the second of which is a fact that will
be used in the proof of Theorem~\ref{thm:quotient}.

\begin{cor} \label{cor:Delta-q-p}
Let $\varV$ be a variety with a Kiss term $q$
and let $\m a\in\varV$.
If $\alpha,\beta\in\Con(\m a)$
satisfy $[\alpha,\beta]=0$, then the following are true.
\begin{enumerate}
\item[(1)]
$    \Mast\alpha\beta=\left\{
    \begin{bmatrix}
    a & q(a,b,c,d)\\
    b & d
    \end{bmatrix}:
    \begin{bmatrix}
    a & c\\
    b & d
    \end{bmatrix}\in R(\alpha,\beta)
    \right\}.
$    
\bigskip    
\item[(2)]
$
  \Mast\alpha\beta=\left\{
\begin{bmatrix}
a & c\\
b & d
\end{bmatrix}\in R(\alpha,\beta)\;:\;
q(a,b,c,d)=c \right\}.
$
\end{enumerate}  
\end{cor}

\begin{proof}
The left-hand side of 
Item~(1) equals the left-hand side of
\eqref{eq:q-and-delta} in
Theorem~\ref{thm:Delta-q-p}~(1).
The right-hand side of 
Item~(1) is shown to be equal to the right-hand side of
\eqref{eq:q-and-delta} in the proof of
Theorem~\ref{thm:Delta-q-p}~(1).

To verify Item~(2), pick any matrix
$M\in\Mast\alpha\beta\,(\subseteq R(\alpha,\beta))$.
According to \eqref{eq:new-rhs},
$M=\begin{bmatrix}
a & c\\
b & d
\end{bmatrix}$
for some 
$c=q(a,b,c_1,d)$ and some 
$\begin{bmatrix}
a & c_1\\
b & d
\end{bmatrix}\in R(\alpha,\beta)$.
By (II)$_q$ and by our assumption $[\alpha,\beta]=0$, we get that
$c=q(a,b,c_1,d)=q(a,b,c,d)$,
which shows that
$M=\begin{bmatrix} a&c\\b&d\end{bmatrix}$ is an element of the right-hand side
of the equality in Item (2).
Conversely, if  
$\begin{bmatrix}
a & c\\
b & d
\end{bmatrix}$
is in the right-hand side of
Item (2), that is,
$\begin{bmatrix}
a & c\\
b & d
\end{bmatrix}\in R(\alpha,\beta)$ and $q(a,b,c,d)=c$, then
$\begin{bmatrix}
a & c\\
b & d
\end{bmatrix}
=\begin{bmatrix}
a & q(a,b,c,d)\\
b & d
\end{bmatrix}\in\Mast\alpha\beta$
follows directly from
Item (1).
\end{proof}

\bigskip

In the final portion of this section,
we address how to remove the
requirement ``$[\alpha,\beta]=0$'' from
our results about $\Mast\alpha\beta$.
Given $\varV$, $\m a\in \varV$, $\alpha,\beta\in\Con(\m a)$,
we consider the natural map
$\nu\colon \m a\to \m a/[\alpha,\beta]\colon a\mapsto a/[\alpha,\beta]$
and denote by $\overline{x}$ and $\overline{X}$
the images $\nu(x)$ and $\nu(X)$ of elements $x\in A$
and subsets or relations $X$ on $A$.
We shall compare $\Mast\alpha\beta$ to
$\Mast{\overline{\alpha}}{\overline{\beta}}$.
We will learn in Theorem~\ref{thm:arbitrary}
that $\Mast\alpha\beta$ is ``$[\alpha,\beta]$-saturated''.
A subset $X\subseteq A$ is \emph{$\theta$-saturated}
for a congruence $\theta\in\Con(\m a)$
if $X$ is a union of $\theta$-classes. 
An
 $n$-ary relation $R\subseteq A^n$ is $\theta$-saturated
if it is a union of $\theta^n$-classes.
The saturation result of Theorem~\ref{thm:arbitrary}~(2)
allows us to reflect
information about $\Mast{\overline{\alpha}}{\overline{\beta}}$,
which is 
the $\Delta$-relation on the quotient $\overline{\m a}$
(where $[\overline{\alpha},\overline{\beta}]=0$ holds),
back to information about
$\Mast\alpha\beta$ on $\m a$
(where $[\alpha,\beta]$ need not be zero).
See e.g.\ Theorem~\ref{thm:quotient}.
The first step is the next result,
which describes properties
of $\Mast\alpha\beta$ that are true in any variety.
(It is possible to derive Theorem~\ref{thm:arbitrary}~(1) from 
the more-general result
\cite[Theorem 4.10 (1)$\Rightarrow$(2)]{moor2021}
by observing that $[\alpha,\beta]\subseteq [\alpha,\beta]_H$,
but we give our own proof here.)

\begin{thm} \label{thm:arbitrary}
Let 
$\m a$ be any algebra
and $\alpha,\beta\in\Con(\m a)$.
\begin{enumerate}
\item If $(a,b)\in[\alpha,\beta]$,  then
  $\begin{bmatrix}a&a\\a&b\end{bmatrix}
\in \Mast\alpha\beta$.
  \smallskip
  
\item $\Mast\alpha\beta$ is $[\alpha,\beta]$-saturated.
\end{enumerate}  
\end{thm}

\begin{proof}
We will freely use the fact that,
since $M(\alpha,\beta)$ is closed under interchanging
rows or columns, the same is true for
$\Mast\alpha\beta$.
Hence the conditions 
$\begin{bmatrix}a&a\\a&b\end{bmatrix}
\in \Mast\alpha\beta$,
$\begin{bmatrix}a&a\\b&a\end{bmatrix}
\in \Mast\alpha\beta$,
$\begin{bmatrix}b&a\\a&a\end{bmatrix}
\in \Mast\alpha\beta$,
and
$\begin{bmatrix}a&b\\a&a\end{bmatrix}
\in \Mast\alpha\beta$ are equivalent.

For Item (1),
let $\delta_0=0\in\Con(\m a)$ be the equality relation.
If $\delta_i$ has been defined, let $\delta_{i+1}$
be the congruence of $\m a$ generated by 
\begin{equation} \label{eq:Xi}
X_{i+1} = \left\{(r,s)\in A^2\;:\;\begin{bmatrix}p&q\\r&s\end{bmatrix}\in M(\alpha,\beta)
\textrm{ for some } (p,q)\in\delta_i\right\}.
\end{equation}
It is easy to prove (see the proof of
\cite[Proposition~4.1(3)]{DemFreVal2019})
that
$\delta_0\leq \delta_1\leq \cdots$ and 
$[\alpha,\beta] = \bigcup \delta_i$.
We execute the proof of Item~(1) by arguing by induction on $i$
that,
\begin{equation} \label{eq:induction}
\tag*{\textrm{Item}~$(1)_i$}
\textrm{  if $(a,b)\in \delta_i$,
then $\begin{bmatrix}a&a\\a&b\end{bmatrix}
  \in \Mast\alpha\beta$.}
\end{equation}  

Item~$(1)_0$ holds since
$\Mast\alpha\beta$ is reflexive.
We assume that Item~$(1)_i$ holds
and proceed to prove that
Item~$(1)_{i+1}$ holds.
The binary relation $X_{i+1}$
defined on line (\ref{eq:Xi})
is a subalgebra of $\m a^2$
since $M(\alpha,\beta)\leq \m a^{2\times 2}$
and $\delta_i\in\Con(\m a)$.
The relation is reflexive and symmetric 
since $M(\alpha,\beta)$ is reflexive
and closed under interchanging columns.
This shows that 
$X_{i+1}$ is a tolerance relation on $\m a$.
Therefore, the congruence $\delta_{i+1}$
generated by $X_{i+1}$
equals the transitive closure of $X_{i+1}$.
If $(a,b)\in\delta_{i+1}$, there must exist a sequence
of matrices
\[
\begin{bmatrix}
p_0& q_0\\
r_0& s_0
\end{bmatrix},
\begin{bmatrix}
p_1& q_1\\
r_1& s_1
\end{bmatrix}, \ldots,
\begin{bmatrix}
p_n& q_n\\
r_n& s_n
\end{bmatrix}, \quad \textrm{with }
\begin{bmatrix} 
p_{j}& q_{j}\\
r_{j}& s_{j}
\end{bmatrix} \in M(\alpha,\beta)
\]
where
$a=r_0$, $b=s_n$, $s_j=r_{j+1}$ for all $j<n$,
and $(p_j, q_j)\in\delta_i$ for all $j$.
We shall employ the following claim
to work through this sequence.

\begin{clm} \label{clm:chain}
If
\begin{enumerate}
\item
  both 
$M=\begin{bmatrix} 
p_{j}& q_{j}\\
r_{j}& s_{j}
\end{bmatrix}$ and 
$N=\begin{bmatrix} 
p_{j+1}& q_{j+1}\\
r_{j+1}& s_{j+1}
\end{bmatrix}$ belong to $\Mast\alpha\beta$, 
\smallskip

\item $s_j=r_{j+1}$,  and
  \smallskip
  
\item $(p_j,q_j)\in\delta_i$,
\end{enumerate}
then $\begin{bmatrix} 
p_{j+1}& q_{j+1}\\
r_{j}& s_{j+1}
\end{bmatrix} \in \Mast\alpha\beta$.
\end{clm}  

Since $(p_j,q_j)\in\delta_i$, we can use the induction
hypothesis to conclude that 
$\begin{bmatrix} 
q_{j}& q_{j}\\
p_{j}& q_{j}
\end{bmatrix}\in\Mast\alpha\beta$.
This is the center-left matrix in
Figure~\ref{fig6}.
The upper-left matrix in Figure~\ref{fig6}
belongs to
$G(\alpha,\beta)\;(\subseteq M(\alpha,\beta)\subseteq \Mast\alpha\beta)$,
since $(q_j,p_{j+1})\in\alpha$.
(The fact that $(q_j,p_{j+1})\in\alpha$ 
follows from
$q_j
\stackrel{\alpha}{\equiv}
s_j = r_{j+1}
\stackrel{\alpha}{\equiv}
p_{j+1}$.)
An examination of Figure~\ref{fig6} establishes that
$\begin{bmatrix} 
p_{j+1}& q_{j+1}\\
r_{j}& s_{j+1}
\end{bmatrix}$
is an element of $\Mast\alpha\beta$, which
concludes the proof of Claim~\ref{clm:chain}.

\medskip

\begin{figure} 
\begin{tikzpicture}[scale=0.9]

\draw[-] (-.3,-.2) -- (2.7,-.2) -- (2.7,1.8) -- (-.3,1.8) -- cycle;
\draw[-] (-.3,2.3) -- (2.7,2.3) -- (2.7,4.3) -- (-.3,4.3) -- cycle;
\draw[-] (-.3,4.8) -- (2.7,4.8) -- (2.7,6.8) -- (-.3,6.8) -- cycle;
\draw[-] (3.2,-.2) -- (6.2,-.2) -- (6.2,6.8) -- (3.2,6.8) -- cycle;

\node at (1.2,.8) {$\Mast\alpha\beta$};
\node at (1.2,3.3) {$\Mast\alpha\beta$};  
\node at (1.2,5.8) {$M(\alpha,\beta)$};  
\node at (4.7,3.3) {$\Mast\alpha\beta$};

\node at (.3,1.5) {$p_j$};
\node at (2.1,1.5) {$q_j$};
\node at (.3,0.1) {$r_j$};
\node at (2.1,0.1) {$s_j$};

\node at (.3,4.0) {$q_j$};
\node at (2.1,4.0) {$q_j$};
\node at (.3,2.6) {$p_j$};
\node at (2.1,2.6) {$q_j$};

\node at (.3,6.5) {$p_{j+1}$};
\node at (2.1,6.5) {$p_{j+1}$};
\node at (.3,5.1) {$q_j$};
\node at (2.1,5.1) {$q_j$};

\node at (3.8,0.1) {$r_{j+1}$};
\node at (5.6,0.1) {$s_{j+1}$};
\node at (3.8,6.5) {$p_{j+1}$};
\node at (5.6,6.5) {$q_{j+1}$};

\draw[-,ultra thick,color=\greencol] (2.7,.05) -- (3.2,.05);
\draw[-,ultra thick,color=\greencol] (2.7,.15) -- (3.2,.15);
\draw[-,ultra thick,color=\greencol] (2.7,6.55) -- (3.2,6.55);
\draw[-,ultra thick,color=\greencol] (2.7,6.45) -- (3.2,6.45);

\draw[-,ultra thick,color=\greencol] (0.25,1.8) -- (0.25,2.3);
\draw[-,ultra thick,color=\greencol] (0.35,1.8) -- (0.35,2.3);
\draw[-,ultra thick,color=\greencol] (2.05,1.8) -- (2.05,2.3);
\draw[-,ultra thick,color=\greencol] (2.15,1.8) -- (2.15,2.3);

\draw[-,ultra thick,color=\greencol] (0.25,4.3) -- (0.25,4.8);
\draw[-,ultra thick,color=\greencol] (0.35,4.3) -- (0.35,4.8);
\draw[-,ultra thick,color=\greencol] (2.05,4.3) -- (2.05,4.8);
\draw[-,ultra thick,color=\greencol] (2.15,4.3) -- (2.15,4.8);

\end{tikzpicture}
\caption{}\label{fig6}
\end{figure}

Applying Claim~\ref{clm:chain} repeatedly, left-to-right, to
the sequence of matrices
\[
\begin{bmatrix}
p_0& q_0\\
r_0& s_0
\end{bmatrix},
\begin{bmatrix}
p_1& q_1\\
r_1& s_1
\end{bmatrix}, \ldots,
\begin{bmatrix}
p_n& q_n\\
r_n& s_n
\end{bmatrix}  \in M(\alpha,\beta)\subseteq \Mast\alpha\beta
\]
leads to 
$\begin{bmatrix}
p_j& q_j\\
r_0& s_j
\end{bmatrix}  \in\Mast\alpha\beta
$
for all $j$, in particular 
$\begin{bmatrix}
p_n& q_n\\
a& b
\end{bmatrix} =
\begin{bmatrix}
p_n& q_n\\
r_0& s_n
\end{bmatrix}  \in\Mast\alpha\beta.$

Apply Claim~\ref{clm:chain}
one more time to the matrices
$M=\begin{bmatrix}
p_n& q_n\\
a& b
\end{bmatrix}  \in\Mast\alpha\beta$
and
$N=\begin{bmatrix}
a& a\\
b& b
\end{bmatrix}  \in\Mast\alpha\beta$.
Note that the assumption that $(a,b)\in\delta_{i+1}\leq [\alpha,\beta]$
implies that $(a,b)\in\alpha$, so
we do indeed have
$\begin{bmatrix}
a& a\\
b& b
\end{bmatrix}  \in\Mast\alpha\beta$.
Claim~\ref{clm:chain} yields that 
$\begin{bmatrix}
a& a\\
a& b
\end{bmatrix}  \in\Mast\alpha\beta$, as desired.
This concludes the inductive proof of Item~(1).

Item~(2) asserts that if
$\begin{bmatrix} a&c\\b&d\end{bmatrix}\in \Mast\alpha\beta$
  and
$(a,a'), (b,b'), (c,c'), (d,d')\in [\alpha,\beta]$,
then
$\begin{bmatrix} a'&c'\\b'&d'\end{bmatrix}\in \Mast\alpha\beta$.
  Since $\Mast\alpha\beta$ is closed under row and column
  interchanges, 
  it suffices to check this one entry at a time,
  so we only explain why 
$\begin{bmatrix} a&c\\b&d\end{bmatrix}\in \Mast\alpha\beta$
and $(a,a')\in [\alpha,\beta]$ imply that 
$\begin{bmatrix} a'&c\\b&d\end{bmatrix}\in \Mast\alpha\beta$.
By Item~(1), the assumption $(a,a')\in [\alpha,\beta]$
implies that
$\begin{bmatrix}a'&a\\a&a\end{bmatrix}\in \Mast\alpha\beta$.
Now Figure~\ref{fig7} shows how to realize 
$\begin{bmatrix} a'&c\\b&d\end{bmatrix}$
  as an element of $\Mast\alpha\beta$.
\end{proof}

\begin{figure}[H]
\begin{tikzpicture}[scale=0.9]

\draw[-] (-.3,-.2) -- (2.7,-.2) -- (2.7,1.8) -- (-.3,1.8) -- cycle;
\draw[-] (-.3,2.3) -- (2.7,2.3) -- (2.7,4.3) -- (-.3,4.3) -- cycle;
\draw[-] (3.2,-.2) -- (6.2,-.2) -- (6.2,4.3) -- (3.2,4.3) -- cycle;

\node at (1.2,.8) {$M(\alpha,\beta)$};
\node at (1.2,3.3) {$\Mast\alpha\beta$};  
\node at (4.7,2.05) {$\Mast\alpha\beta$};

\node at (.3,4.0) {$a'$};
\node at (2.1,4.0) {$a$};
\node at (.3,2.6) {$a$};
\node at (2.1,2.6) {$a$};

\node at (.3,1.5) {$a$};
\node at (2.1,1.5) {$a$};
\node at (.3,0.1) {$b$};
\node at (2.1,0.1) {$b$};

\node at (3.8,0.1) {$b$};
\node at (5.6,0.1) {$d$};
\node at (3.8,4.0) {$a$};
\node at (5.6,4.0) {$c$};

\draw[-,ultra thick,color=\greencol] (2.7,.05) -- (3.2,.05);
\draw[-,ultra thick,color=\greencol] (2.7,.15) -- (3.2,.15);
\draw[-,ultra thick,color=\greencol] (2.7,4.05) -- (3.2,4.05);
\draw[-,ultra thick,color=\greencol] (2.7,3.95) -- (3.2,3.95);

\draw[-,ultra thick,color=\greencol] (0.25,1.8) -- (0.25,2.3);
\draw[-,ultra thick,color=\greencol] (0.35,1.8) -- (0.35,2.3);
\draw[-,ultra thick,color=\greencol] (2.05,1.8) -- (2.05,2.3);
\draw[-,ultra thick,color=\greencol] (2.15,1.8) -- (2.15,2.3);

\end{tikzpicture}
\caption{}\label{fig7}  
\end{figure} 

The converse
of Theorem~\ref{thm:arbitrary}~(1) is not true for arbitrary varieties.
A counterexample can be constructed by the technique
described in \cite[Example~3.2]{kearnes-kiss2013}.
Nevertheless, the converse
of Theorem~\ref{thm:arbitrary}~(1) is true
if the variety has a Kiss term.
This is part of the next theorem.

\begin{thm} \label{thm:quotient}
Let $\varV$ be a variety with a Kiss term $q$, $\m a\in\varV$,
and $\alpha,\beta\in\Con(\m a)$.
\begin{enumerate}
\item
$\Mast\alpha\beta =
  \left\{\begin{bmatrix}a&c\\b&d\end{bmatrix}\in R(\alpha,\beta)
  \;:\;
q(a,b,c,d)\stackrel{[\alpha,\beta]}{\equiv} c\right\}.  
$
\item
If $\begin{bmatrix}a&c\\b&d
\end{bmatrix} \in R(\alpha,\beta)$ and
$c'=q(a,b,c,d)$, then
$\begin{bmatrix} a&c'\\b&d\end{bmatrix} \in \Mast\alpha\beta$.\\
(Lemma~\ref{lm:crucial} is true without
  the assumption ``$[\alpha,\beta]=0$''.)
\item  \label{quotient:it4}
$
[\alpha,\beta] =
\left\{(a,b)\in A^2\;:\;\begin{bmatrix}a&a\\a&b\end{bmatrix}\in \Mast\alpha\beta\right\}
$.\\
(The converse
of Theorem~\ref{thm:arbitrary}~(1) is true in the presence of
a Kiss term.)
\end{enumerate}
\end{thm}

\begin{proof}
  To simplify notation, we will write $\ol{\m a}$ for $\m a/[\alpha,\beta]$,
  and we will use the ``bar notation'' introduced in the paragraph preceding
  Theorem~\ref{thm:arbitrary}
  for passing from elements of $\m a$ or relations on
  $\m a$ to the elements of $\ol{\m a}$ or relations on $\ol{\m a}$
  under the quotient map $\m a\to\ol{\m a}$.

  For the equality in Item~(1)
  we need to argue that the following are equivalent
  for arbitrary elements $a,b,c,d$ of $\m a$:
\begin{enumerate}
\item[(a)]
  $\begin{bmatrix}a&c\\b&d\end{bmatrix}\in \Mast\alpha\beta$.
\smallskip
\item[(b)]
  $\begin{bmatrix}a&c\\b&d\end{bmatrix}\in R(\alpha,\beta)$
  and $q(a,b,c,d)\stackrel{[\alpha,\beta]}{\equiv} c$.
\end{enumerate}  
We will establish (a)~$\Leftrightarrow$~(b) by showing that each one of
(a) and (b) is equivalent to the following condition:
\begin{enumerate}
\item[(c)]
  $\begin{bmatrix}\ol{a}&\ol{c}\\ \ol{b}&\ol{d}\end{bmatrix}\in
  \Mast{\ol{\alpha}}{\ol{\beta}}$.
\end{enumerate}  

  We start with (a)~$\Leftrightarrow$~(c). 
  Proving (a)~$\Rightarrow$~(c) is equivalent to
  showing that
  $\ol{\Mast\alpha\beta}\subseteq\Mast{\ol{\alpha}}{\ol{\beta}}$.
  To prove this inclusion,
  recall that in the proof of Lemma~\ref{lm:subalg} we defined an increasing
  sequence
  \begin{equation*}
    M(\alpha,\beta)
    =:M_0(\alpha,\beta)
    \subseteq M_1(\alpha,\beta)
    \subseteq\dots
    \subseteq M_n(\alpha,\beta)
    \subseteq\dots
  \end{equation*}
  of subalgebras of $\m A^{2\times2}$ and proved that
  $\Mast\alpha\beta=\bigcup_{n\ge0} M_n(\alpha,\beta)$.
  Applying this fact to the algebra
  $\ol{\m a}$ and
  $\ol{\alpha},\ol{\beta}\in\Con(\ol{\m a})$,
  we get a sequence
  \begin{equation*}
    M(\ol{\alpha},\ol{\beta})
    =:M_0(\ol{\alpha},\ol{\beta})
    \subseteq M_1(\ol{\alpha},\ol{\beta})
    \subseteq\dots
    \subseteq M_n(\ol{\alpha},\ol{\beta})
    \subseteq\dots
  \end{equation*}
  of subalgebras of $\ol{\m a}^{2\times2}$ such that
  $\Mast{\ol{\alpha}}{\ol{\beta}}=\bigcup_{n\ge0} M_n(\ol{\alpha},\ol{\beta})$.
  We claim that
  \begin{equation}\label{eq:nu}
    \ol{M_n(\alpha,\beta)}\subseteq M_n(\ol{\alpha},\ol{\beta})
  \end{equation}  
  holds for all $n\ge0$.
  For $n=0$, we have equality, because
  $\ol{G(\alpha,\beta)}=G(\ol{\alpha},\ol{\beta})$ for the generating sets
  of $M_0(\alpha,\beta)=M(\alpha,\beta)$ and
  $M_0(\ol{\alpha},\ol{\beta})=M(\ol{\alpha},\ol{\beta})$ (see
  the definition of $G(\alpha,\beta)$
  at \eqref{eq:G}).
  For $n>0$, the definitions of $M_n(\alpha,\beta)$ and
  $M_n(\ol{\alpha},\ol{\beta})$ yield \eqref{eq:nu} by induction.
  Thus,
  \begin{equation*}
    \ol{\Mast\alpha\beta}
    = \bigcup_{n\ge0} \ol{M_n(\alpha,\beta)}
    \subseteq \bigcup_{n\ge0} M_n(\ol{\alpha},\ol{\beta})
    =\Mast{\ol{\alpha}}{\ol{\beta}},
  \end{equation*}    
  as claimed.

  For the converse, (c)~$\Rightarrow$~(a), we need to argue that if
  $\begin{bmatrix}\ol{a}&\ol{c}\\\ol{b}&\ol{d}\end{bmatrix}\in
    \Mast{\ol{\alpha}}{\ol{\beta}}$,
  then
    $\begin{bmatrix}a&c\\b&d\end{bmatrix}\in \Mast\alpha\beta$.
  Since
  $\Mast{\ol{\alpha}}{\ol{\beta}}=\bigcup_{n\ge0} M_n(\ol{\alpha},\ol{\beta})$,
  it suffices to show that for every $n\ge0$,
  \begin{equation}\label{eq:induction2}
\forall a, b, c, d\in A,\;\;  \text{if\ \ 
  $\begin{bmatrix}\ol{a}&\ol{c}\\\ol{b}&\ol{d}\end{bmatrix}\in
  M_n(\ol{\alpha},\ol{\beta})$,\ \ 
  then\ \ 
  $\begin{bmatrix}a&c\\b&d\end{bmatrix}\in \Mast\alpha\beta$.}
  \end{equation}
  We will proceed by induction on $n$.

  First let $n=0$ and choose any
$\begin{bmatrix}\ol{a}&\ol{c}\\\ol{b}&\ol{d}\end{bmatrix}\in
  M_0(\ol{\alpha},\ol{\beta})=\ol{M_0(\alpha,\beta)}$.
There must exist
  $\begin{bmatrix}a_0& c_0\\ b_0& d_0\end{bmatrix}\in M_0(\alpha,\beta)$
  such that $\ol{a_0}=\ol{a}$, $\ol{b_0}=\ol{b}$, $\ol{c_0}=\ol{c}$,
  and $\ol{d_0}=\ol{d}$.
  Thus,
  $\begin{bmatrix}a_0& c_0\\ b_0& d_0\end{bmatrix}\in \Mast\alpha\beta$
  and
  $a_0\stackrel{[\alpha,\beta]}{\equiv}a$,
  $b_0\stackrel{[\alpha,\beta]}{\equiv}b$,
  $c_0\stackrel{[\alpha,\beta]}{\equiv}c$,
  $d_0\stackrel{[\alpha,\beta]}{\equiv}d$.
Since $\Mast\alpha\beta$ is $[\alpha,\beta]$-saturated  
(Theorem~\ref{thm:arbitrary}~(2)) we get that
  $\begin{bmatrix}a&c\\b&d\end{bmatrix}\in \Mast\alpha\beta$,
    proving \eqref{eq:induction2} for $n=0$.

Next, consider the case when $n$ is odd, and assume
that \eqref{eq:induction2} holds for $n-1$ in place of $n$.   
By definition, if  
$\begin{bmatrix}\ol{a}&\ol{c}\\\ol{b}&\ol{d}\end{bmatrix}\in
M_n(\ol{\alpha},\ol{\beta})$, then
$\ol{\m a}$ has elements $\ol{r},\ol{s}$ ($r,s\in A$)
such that
$\begin{bmatrix}\ol{a}&\ol{r}\\\ol{b}&\ol{s}\end{bmatrix},
\begin{bmatrix}\ol{r}&\ol{c}\\\ol{s}&\ol{d}\end{bmatrix}\in
M_{n-1}(\ol{\alpha},\ol{\beta})$.
The induction hypothesis yields that
$\begin{bmatrix} a & r\\ b & s\end{bmatrix}, 
\begin{bmatrix} r & c \\ s & d \end{bmatrix}\in \Mast\alpha\beta$.
Thus,
$\begin{bmatrix}a&c\\b&d\end{bmatrix}\in \Mast\alpha\beta$,
proving \eqref{eq:induction2} when $n$ is odd.

The case when $n>0$ is even can be handled the same way,
so the proof of
\eqref{eq:induction2} is complete.
This finishes the proof that (a)~$\Leftrightarrow$~(c).

It remains to show that (b)~$\Leftrightarrow$~(c).
Since $R(\alpha,\beta)$ is $(\alpha\cap\beta)$-saturated and
$[\alpha,\beta]\le(\alpha\cap\beta)$,
$R(\alpha,\beta)$ is $[\alpha,\beta]$-saturated.
Therefore, Condition~(b) is equivalent (after factoring by $[\alpha,\beta]$) to
\begin{enumerate}
\item[($\ol{\textrm{b}}$)]
  $\begin{bmatrix}\ol{a}&\ol{c}\\ \ol{b}&\ol{d}\end{bmatrix}\in
  R(\ol{\alpha},\ol{\beta})$
  and $q(\ol{a},\ol{b},\ol{c},\ol{d})=\ol{c}$.
\end{enumerate}  
Since $[\ol{\alpha},\ol{\beta}]=0$, 
Corollary~\ref{cor:Delta-q-p}~(2)
proves that
Condition~($\ol{\textrm{b}}$) is equivalent to
$\begin{bmatrix}\ol{a}&\ol{c}\\ \ol{b}&\ol{d}\end{bmatrix}\in
\Mast{\ol{\alpha}}{\ol{\beta}}$, which is Condition (c).
This finishes the proof of Item~(1).

To prove Item~(2) assume that
$\begin{bmatrix}a&c\\b&d
\end{bmatrix} \in R(\alpha,\beta)$ and
$q(a,b,c,d)=c'$. Factoring by $[\alpha,\beta]$
yields
$\begin{bmatrix}\ol{a}&\ol{c}\\ \ol{b}&\ol{d}\end{bmatrix}\in
R(\ol{\alpha},\ol{\beta})$
and $q(\ol{a},\ol{b},\ol{c},\ol{d})=\ol{c}'$,
but in this quotient $[\ol{\alpha},\ol{\beta}]=0$.
By Lemma~\ref{lm:crucial},
$\begin{bmatrix}\ol{a}&\ol{c}'\\ \ol{b}&\ol{d}\end{bmatrix}\in
\Mast{\ol{\alpha}}{\ol{\beta}}$.
Using the equivalence (a)~$\Leftrightarrow$~(c) from
the proof of Item~(1), we get
$\begin{bmatrix}a&c'\\b&d
\end{bmatrix} \in \Mast\alpha\beta$.

For Item~\eqref{quotient:it4}, we have $\subseteq$ from
Lemma~\ref{thm:arbitrary}~(1) and
 we need to prove 
the opposite inclusion.
Assume $\begin{bmatrix}a&a\\a&b\end{bmatrix}\in\Mast\alpha\beta$.
Since $\Mast\alpha\beta$ is closed under interchanging rows,
we also have that
$\begin{bmatrix}a&b\\a&a\end{bmatrix}\in\Mast\alpha\beta$. 
Now the description of $\Mast\alpha\beta$ in Item~(1) implies that
$\begin{bmatrix}a&b\\a&a\end{bmatrix}\in R(\alpha,\beta)$
and $q(a,a,b,a)\stackrel{[\alpha,\beta]}{\equiv}b$.
Since $\begin{bmatrix}a&a\\a&a\end{bmatrix}\in\Mast\alpha\beta$,
properties (I)$_q$ and (II)$_q$ of $q$ force
$a=q(a,a,a,a)\stackrel{[\alpha,\beta]}{\equiv}q(a,a,b,a)$.
Thus, $a\stackrel{[\alpha,\beta]}{\equiv}b$, which completes the
proof of equality in Item~\eqref{quotient:it4}.
Note that this also proves the final statement in Item~\eqref{quotient:it4}. 
\end{proof}

\begin{rem}
  
1. Ralph Freese and Ralph McKenzie 
\cite[Theorem~4.9]{fremck1987} proved our
Theorem~\ref{thm:quotient}\eqref{quotient:it4}
in the restricted setting of congruence modular
varieties,
but with $\Mast\alpha\beta$ replaced by $\Delta_{\alpha,\beta}$
(see Section~\ref{sec:summary}).  It can be shown that
$\Delta_{\alpha,\beta}=\Mast\alpha\beta$ in any congruence modular variety,
so our Theorem~\ref{thm:quotient}~\eqref{quotient:it4}
can be viewed as an extension of their result
to varieties with a difference term.

\smallskip
2.
Similarly, Kiss 
\cite[Theorem~3.8 (ii)]{kiss1992} proved our 
Theorem~\ref{thm:quotient} Items~(1) and~(2) in the congruence modular setting,
but with $\Mast\alpha\beta$ replaced by $\Delta_{\alpha,\beta}$.

\smallskip
3.  Our Theorem~\ref{thm:quotient}~\eqref{quotient:it4} can
be deduced from
a related result of Moorhead.  In \cite{moor2021}, Moorhead defines the
``hypercommutator'' $[\alpha,\beta]_H$
of congruences $\alpha$ and $\beta$ to be the smallest congruence 
$\delta$ satisfying the implication $(a,c) \in \delta \implies (b,d) \in 
\delta$ for all $\begin{bmatrix}a&c\\b&d\end{bmatrix} \in \Mast\alpha\beta$.
In \cite[Proposition~3.5]{moor2021}, Moorhead
shows that if $\alpha,\beta$ are congruences of any algebra whatsoever, then 
$[\alpha,\beta]_H$ is equal to the set on the right-hand side of the
equation in the statement of our
Theorem~\ref{thm:quotient}~\eqref{quotient:it4}.
Since $[\alpha,\beta]_H$ is always sandwiched between the usual commutator
$[\alpha,\beta]$ and the linear commutator $[\alpha,\beta]_\ell$
(see \cite{ks1998} for the definition of $[\alpha,\beta]_\ell$),
and since $[\alpha,\beta]=[\alpha,\beta]_\ell$ in varieties with a difference
term by \cite[Lemma~2.2]{kea1995} and \cite[Corollary~4.5]{ks1998}, 
our Theorem~\ref{thm:quotient}~\eqref{quotient:it4} can be viewed as
the specialization of Moorhead's
\cite[Proposition~3.5]{moor2021} to algebras in varieties with a difference
term.
\end{rem}

We can use Theorem~\ref{thm:quotient} to obtain the following
result, which was previously obtained by Moorhead
\cite[Theorem~5.2]{moor2021} in the special
case that $\alpha=\beta$.

\begin{cor} \label{thm:sdmeet}
Let $\varV$ be a congruence meet-semidistributive variety.
If $\m a\in\varV$ and
$\alpha,\beta\in\Con(\m a)$, then $R(\alpha,\beta)=\Mast\alpha\beta$.
\end{cor}

\begin{proof}
If $\varV$ is congruence meet-semidistributive, then
$q(x,y,z,w)=z$ is a Kiss term for $\varV$. Using this
and the fact that $[\alpha,\beta]$ is a reflexive relation,
the statement 
of Theorem~\ref{thm:quotient}~(1)
reduces to $R(\alpha,\beta)=\Mast\alpha\beta$.
\end{proof}

In closing, we note that Kiss's proof of Theorem~\ref{thm:main} 
in the congruence
modular case was relatively short and used standard properties of congruences
and the commutator operation in congruence modular varieties.  By contrast,
our proof of Theorem~\ref{thm:main} is long, complicated and syntactic.

\medskip\noindent\textsc{Problem.}  
Develop the theory of congruences and
commutators in varieties with a difference term, sufficient to support
a short, direct proof of our Theorem~\ref{thm:main}.

\bibliography{kisstermbib}
\end{document}